\documentclass[11pt]{article}

\title{Zeta functions of dynamically tame Liouville domains}
\author{Michael Hutchings\footnote{Partially supported by NSF grant DMS-2005437.}}
\date{}

\usepackage{amssymb}
\usepackage{latexsym}
\usepackage{amsmath}
\usepackage{amsthm}
\usepackage{amscd}
\usepackage{tikz-cd}

\numberwithin{equation}{section}

\newtheorem{theorem}{Theorem}[section]

\newtheorem{proposition}[theorem]{Proposition}
\newtheorem{corollary}[theorem]{Corollary}
\newtheorem{lemma}[theorem]{Lemma}
\newtheorem{lemma-definition}[theorem]{Lemma-Definition}

\newtheorem{question}[theorem]{Question}

\theoremstyle{definition}
\newtheorem{definition}[theorem]{Definition}

\newtheorem{remark}[theorem]{Remark}

\newtheorem{example}[theorem]{Example}

\newtheorem{notation}[theorem]{Notation}

\newcommand{\floor}[1]{\left\lfloor #1 \right\rfloor}

\newcommand{\C}{{\mathbb C}}
\newcommand{\Q}{{\mathbb Q}}
\newcommand{\R}{{\mathbb R}}

\newcommand{\Z}{{\mathbb Z}}

\newcommand{\op}{\operatorname}

\newcommand{\Ker}{\op{Ker}}

\newcommand{\bpm}{\begin{pmatrix}}
\newcommand{\epm}{\end{pmatrix}}

\renewcommand{\epsilon}{\varepsilon}

\begin{document}

\maketitle

\begin{abstract}
We define a dynamical zeta function for nondegenerate Liouville domains, in terms of Reeb dynamics on the boundary. We use filtered equivariant symplectic homology to (i) extend the definition of the zeta function to a more general class of ``dynamically tame'' Liouville domains, and (ii) show that the zeta function of a dynamically tame Liouville domain is invariant under exact symplectomorphism of the interior. As an application, we find examples of open domains in $\R^4$, arbitrarily close to a ball, which are not symplectomorphic to open star-shaped toric domains.
\end{abstract}

\tableofcontents

\setcounter{tocdepth}{2}


\section{Introduction}

\subsection{A motivating question}
\label{sec:motivating}

In 2022, the author heard the following question from Jean Gutt and Vinicius Ramos:

\begin{question}
\label{question:motivating} Is every bounded open convex set in $\R^4$ with smooth boundary symplectomorphic to the interior of a star-shaped toric domain?
\end{question}

To clarify what this means, if $\Omega\subset\R_{\ge 0}^2$, define the associated {\bf toric domain\/}
\[
X_\Omega = \left\{z\in\C^2 \;\big|\; \left(\pi|z_1|^2,\pi|z_2|^2\right) \in\Omega\right\},
\]
with the restriction of the standard symplectic form on $\C^2=\R^4$ defined by
\[
\omega_0 = \sum_{i=1}^2 dx_i\wedge dy_i.
\]

\begin{definition}
\label{def:osstd}
The set $X_\Omega\subset\C^2$ is a {\bf star-shaped toric domain\/} if $\Omega$ is a compact set in $\R^2_{\ge 0}$ whose interior contains the origin, whose boundary is a smooth curve $\partial_+\Omega$ from a point $(a,0)$ with $a>0$ to a point $(0,b)$ with $b>0$ which is transverse to the radial vector field on $\R^2$.
\end{definition}

For example, if $\partial_+\Omega$ is the line segment from $(a,0)$ to $(0,b)$, then $X_\Omega$ is the ellipsoid
\[
E(a,b) = \left\{z\in\C^2 \;\bigg|\; \frac{\pi|z_1|^2}{a} + \frac{\pi|z_2|^2}{b} \le 1\right\}.
\]
Star-shaped toric domains are basic examples for studying symplectic embedding questions using various technology, see e.g.\ \cite{concaveconvex,beyond,siegel}.
The question is now: If $U\subset\R^4=\C^2$ is a bounded convex open set with smooth boundary, equipped with the restriction of the standard symplectic form $\omega_0$, is $U$ symplectomorphic to the interior of a star-shaped toric domain $X_\Omega$ as above?

One motivation for this question was that an affirmative answer would imply the strong Viterbo conjecture in four dimensions\footnote{Haim-Kislev and Ostrover \cite{hko} recently found a counterexample to the Viterbo conjecture.}, asserting that all normalized symplectic capacities agree on convex domains in $\R^4$. The reason is that if a bounded open convex set $U$ is symplectomorphic to the interior of a star-shaped toric domain $X_\Omega$, then $X_\Omega$ is necessarily dynamically convex (because of the convexity of $U$ together with Remark~\ref{rem:90s} below), and it is shown in \cite[Thm.\ 1.7]{ghr} that all normalized symplectic capacities agree on dynamically convex toric domains in $\R^4$.

However, it turns out that the answer to Question~\ref{question:motivating} is no. Here is one way to find counterexamples, which follows quickly from the machinery developed in this paper. Define a {\bf smooth star-shaped domain\/} in $\R^{2n}$ to be a compact domain $X\subset\R^{2n}$ whose boundary is a connected smooth hypersurface transverse to the radial vector field. In this situation, the standard Liouville form on $\R^{2n}$ given by
\begin{equation}
\label{eqn:standardLiouvilleform}
\lambda_0 = \frac{1}{2}\sum_{i=1}^n\left(x_i\,dy_i-y_i\,dx_i\right)
\end{equation}
restricts to a contact form on $\partial X$. This determines a Reeb vector field on $\partial X$; see \S\ref{sec:zetanondeg} below for definitions.

\begin{proposition}
\label{prop:motivating}
(proved in \S\ref{sec:introexamples})
Let $X\subset\R^4$ be a smooth star-shaped domain. Suppose that the Reeb flow on $\partial X$ is nondegenerate\footnote{Instead of assuming that the Reeb flow on $\partial X$ is nondegenerate, it is enough to assume that the Reeb orbit $\gamma$ is nondegenerate and that its symplectic action is not an accumulation point of the action spectrum, so that the Euler characteristic jump in Definition~\ref{def:tameecj} below is well defined and negative.}. Suppose that there exists a simple positive hyperbolic Reeb orbit $\gamma$ in $\partial X$ whose symplectic action does not agree with the total symplectic action of any finite set of Reeb orbits in $\partial X$, other than $\{\gamma\}$ itself. Then $\op{int}(X)$ is not symplectomorphic to the interior of a star-shaped toric domain\footnote{Another approach to distinguishing open star-shaped domains from convex or concave toric domains in arbitrary dimension was recently presented in \cite[Cor.\ 2.8]{bg}.}.
\end{proposition}

\begin{example}
One can construct domains $X$ which satisfy the conditions in Proposition~\ref{prop:motivating}, and which are arbitrarily close in any number of derivatives to a round ball (and in particular convex), as follows. Let $f:\C P^1\to\R$ be a $C^2$-small Morse function. Define
\begin{equation}
\label{eqn:Xf}
X_f = \left\{z\in \C^2\setminus\{0\} \;\bigg|\; \pi|z|^2 \le e^{f([z_1:z_2])} \right\} \cup \{0\}.
\end{equation}
Then $\partial X_f$ has one simple Reeb orbit for each critical point $[z_1:z_2]$ of $f$, given by a scaling of the Hopf circle over $[z_1:z_2]$; see the calculations in \S\ref{sec:computations}. There may be additional simple Reeb orbits of much greater symplectic action, which appear in $S^1$-families; a small additional perturbation of $X_f$ will make these nondegenerate, cf.\ \cite[\S6]{hwz}. If $f$ has a unique index $1$ critical point, then the corresponding simple Reeb orbit $\gamma$ is a positive hyperbolic Reeb orbit satisfying the symplectic action condition in Proposition~\ref{prop:motivating}. See \S\ref{sec:computations} for more details.
\end{example}

The proof of Proposition~\ref{prop:motivating} uses information about Reeb orbits coming from filtered equivariant symplectic homology. The data from filtered equivariant symplectic homology that we need can be conveniently encoded in a kind of ``zeta function''. The main goal of this paper is to define this zeta function for suitable Liouville domains and prove that it is invariant under exact symplectomorphisms of open Liouville domains.

\subsection{Zeta functions of nondegenerate contact manifolds}
\label{sec:zetanondeg}

To start, we now recall some basic definitions and define a dynamical zeta function for a nondegenerate contact manifold.

Let $Y$ be a closed $2n-1$ dimensional manifold. A {\bf contact form\/} on $Y$ is a $1$-form $\lambda$ on $Y$ such that $\lambda\wedge(d\lambda)^{n-1}\neq0$. Given a contact form $\lambda$, we define the {\bf contact structure\/} $\xi=\Ker(\lambda)$, which is a symplectic vector bundle with the fiberwise symplectic form $d\lambda$. The contact form $\lambda$ also determines the {\bf Reeb vector field\/} $R$ characterized by $d\lambda(R,\cdot)=0$ and $\lambda(R)=1$. A {\bf Reeb orbit\/} is a periodic orbit of $R$, i.e.\ a map $\gamma:\R/T\Z\to Y$ for some $T>0$ such that $\gamma'=R\circ\gamma$. We declare two Reeb orbits to be equivalent if they differ by precomposition with translations of the domain. We define the {\bf symplectic action\/} $\mathcal{A}(\gamma)>0$ to be the period $T$. A Reeb orbit $\gamma$ is {\bf simple\/} if the map $\gamma$ is an embedding. Every Reeb orbit $\gamma$ is a covering of a simple Reeb orbit, and we denote the covering multiplicity by $d(\gamma)\in\Z^{>0}$, so that $\gamma$ is simple if and only if $d(\gamma)=1$. We denote the set of all Reeb orbits by $\mathcal{P}(Y,\lambda)$, and the set of simple Reeb orbits by $\mathcal{P}_{\op{simp}}(Y,\lambda)$.

Given a Reeb orbit $\gamma:\R/T\Z\to Y$, the derivative of the time $T$ Reeb flow restricts to a symplectic linear map
\[
P_\gamma: \xi_{\gamma(0)} \longrightarrow \xi_{\gamma(0)},
\]
which we call the {\bf linearized return map\/}. We say that the Reeb orbit $\gamma$ is {\bf nondegenerate\/} if the linearized return map $P_\gamma$ does not have $1$ as an eigenvalue. In this case, we define the {\bf Lefschetz sign\/} $(-1)^{\epsilon(\gamma)}\in\{\pm 1\}$ to be the sign of $\det(1-P_\gamma)$. Here we regard $\epsilon(\gamma)\in\Z/2$. We say that the contact form $\lambda$ is {\bf nondegenerate\/} if all Reeb orbits are nondegenerate. 
A generic contact form on $Y$ is nondegenerate.

If $R$ is a commutative ring, let $\Lambda_R$ denote the {\bf universal Novikov ring\/} over $R$, consisting of functions $\alpha:\R\to R$ such that for every $L\in\R$, the restriction of $\alpha$ to the interval $(-\infty,L)$ is finitely supported. We typically write such a function as a formal power series
\[
\sum_{s\in\R}\alpha(s)t^s.
\]
The addition on $\Lambda_R$ is defined by addition of functions, and the product is defined by the convolution
\[
(\alpha\beta)(s) = \sum_{s'\in\R}\alpha(s')\beta(s-s').
\]
Our default choice of ring is $R=\Z$, and we write $\Lambda=\Lambda_\Z$.

Let $\Lambda_R^+$ denote the set of $\alpha\in\Lambda_R$ supported on $(0,\infty)$. There is an exponential map
\[
\exp:\Lambda_\Q^+ \longrightarrow \Lambda_\Q
\]
defined by
\[
\exp(\alpha) = \sum_{k=0}^\infty \frac{\alpha^k}{k!}.
\]

\begin{definition}
Let $\lambda$ be a nondegenerate contact form on a closed odd dimensional manifold $Y$. Define the {\bf zeta function\/}
\begin{equation}
\label{eqn:defzeta}
\zeta(Y,\lambda) = \exp\sum_{\gamma\in\mathcal{P}(Y,\lambda)}\frac{(-1)^{\epsilon(\gamma)}}{d(\gamma)}t^{\mathcal{A}(\gamma)} \in \Lambda_\Q.
\end{equation}
\end{definition}

Note that $\zeta(Y,\lambda)$ is well defined, because nondegeneracy of $\lambda$ and compactness of $Y$ imply that for every $L\in\R$ there are only finitely many Reeb orbits $\gamma$ with action $\mathcal{A}(\gamma)<L$, so that the sum that we are exponentiating in \eqref{eqn:defzeta} is in $\Lambda_\Q^+$.

In fact, the zeta function has integer coefficients, i.e.\ $\zeta(Y,\lambda)\in\Lambda$. This is a consequence of the following product formula. To state the formula, if $\gamma\in\mathcal{P}_{\op{simp}}(Y,\lambda)$ is a simple Reeb orbit and if $d$ is a positive integer, let $\gamma^d$ denote the $d$-fold cover of $\gamma$.
Also note that the units $\Lambda^\times\subset\Lambda$ consist of elements with leading coefficient $\pm1$, and in particular $1+\Lambda^+\subset\Lambda^\times$. If $\alpha\in\Lambda^+$ then we have
\begin{equation}
\label{eqn:novikovinverse}
(1-\alpha)^{-1}=\sum_{k=0}^\infty \alpha^k.
\end{equation}

\begin{lemma}
\label{lem:productformula}
(proved in \S\ref{sec:algebraic})
Let $\lambda$ be a nondegenerate contact form on a closed odd dimensional manifold $Y$. Then
\begin{equation}
\label{eqn:productformula}
\zeta(Y,\lambda)=\prod_{\gamma\in\mathcal{P}_{\op{simp}}(Y,\lambda)}
\left(1-(-1)^{\epsilon(\gamma)+\epsilon(\gamma^2)} t^{\mathcal{A}(\gamma)}\right)^{-(-1)^{\epsilon(\gamma^2)}} \in 1 + \Lambda^+.
\end{equation}
\end{lemma}

\begin{example}
\label{ex:dim3}
Suppose $\dim(Y)=3$, and continue to assume that the contact form $\lambda$ is nondegenerate. For a simple Reeb orbit $\gamma$, since $P_\gamma$ is a rank $2$ symplectic map, there are three possibilities for the eigenvalues of $P_\gamma$: We say that $\gamma$ is {\bf positive hyperbolic\/} if the eigenvalues of $P_\gamma$ are positive, {\bf negative hyperbolic\/} if the eigenvalues of $P_\gamma$ are negative, and {\bf elliptic\/} if the eigenvalues of $P_\gamma$ are on the unit circle. By the product formula of Lemma~\ref{lem:productformula}, see \eqref{eqn:iterateparity}, we have
\begin{equation}
\label{eqn:3dimproduct}
\zeta(Y,\lambda) = \prod_{\gamma\in\mathcal{P}_{\op{simp}}(Y,\lambda)}\left\{\begin{array}{cl}
\left(1-t^{\mathcal{A}(\gamma)}\right)^{-1}, & \mbox{$\gamma$ elliptic\/},\\
1-t^{\mathcal{A}(\gamma)}, & \mbox{$\gamma$ positive hyperbolic\/},\\
1+t^{\mathcal{A}(\gamma)}, & \mbox{$\gamma$ negative hyperbolic\/}.
\end{array}\right.
\end{equation}
\end{example}

\begin{remark}
\label{rem:ech}
In the situation of Example~\ref{ex:dim3}, an {\bf ECH generator\/} is a finite set of pairs $\alpha=\{(\alpha_i,m_i)\}$ where the $\alpha_i$ are distinct simple Reeb orbits, the $m_i$ are positive integers, and $m_i=1$ whenever $\alpha_i$ is (positive or negative) hyperbolic. The symplectic action of the ECH generator $\alpha$ is defined by
\[
\mathcal{A}(\alpha) = \sum_i m_i \mathcal{A}(\alpha_i).
\]
We also define the mod 2 grading of $\alpha$, denoted by $|\alpha|\in\Z/2$, to be the parity of the number of $i$ such that $\alpha_i$ is positive hyperbolic. The embedded contact homology of $(Y,\lambda)$ is the homology of a chain complex which is generated by ECH generators, and which has a canonical mod 2 grading as above; see e.g.\ \cite{bn}.

It follows by expanding the product \eqref{eqn:3dimproduct} that if $\mathcal{P}_{\op{ECH}}(Y,\lambda)$ denotes the set of ECH generators, then
\begin{equation}
\label{eqn:ECHzeta}
\zeta(Y,\lambda) = \sum_{\alpha\in\mathcal{P}_{\op{ECH}}(Y,\lambda)}(-1)^{|\alpha|}t^{\mathcal{A}(\alpha)}.
\end{equation}
In fact, this formula is closely related to the motivation for the definition of embedded contact homology, as explained in \cite[\S2.6]{bn}. Although embedded contact homology is only defined in three dimensions, part of the motivation for this paper was to extend \eqref{eqn:ECHzeta} to obtain invariants of higher dimensional contact manifolds and Liouville domains.
\end{remark}

\begin{remark}
Various kinds of dynamical zeta functions have been extensively studied at least since the 1940s. Zeta functions like \eqref{eqn:defzeta} are often considered without the sign $(-1)^{\epsilon(\gamma)}$; see e.g.\ the survey by Ruelle \cite{ruelle}. Zeta functions with signs as in \eqref{eqn:defzeta}, with relations to Reidemeister torsion, have been studied for example by Fried \cite{fried}, with ideas going back to Weil \cite{weil} and Milnor \cite{milnor}. Similar zeta functions in the context of Morse theory for circle-valued functions and closed 1-forms were studied in \cite{rt1,pajitnov,rt3}. A different zeta function from ECH with fewer terms than \eqref{eqn:ECHzeta} is studied in \cite{cgs}.
\end{remark}


\subsection{Zeta functions of nondegenerate Liouville domains}

We now discuss the extent to which the interior of a nondegenerate Liouville domain sees the zeta function of its boundary.

\begin{definition}
\label{def:nld}
A {\bf Liouville domain\/} is a pair $(X,\lambda)$ where $X$ is a compact connected $2n$-dimensional manifold with boundary, and $\lambda$ is a 1-form on $X$ such that:
\begin{itemize}
\item $d\lambda$ is a symplectic form on $X$.
\item
$\lambda|_{\partial X}$ is a contact form on $\partial X$.
\item
The orientation of the contact manifold $(\partial X,\lambda|_{\partial X})$ agrees with the boundary orientation of the symplectic manifold $(X,d\lambda)$.
\end{itemize}
We say that the Liouville domain $(X,\lambda)$ is {\bf nondegenerate\/} if the contact form $\lambda|_{\partial X}$ is nondegenerate.
\end{definition}

For example, if $X$ is a smooth star-shaped domain in $\R^{2n}$, then $(X,\lambda)$ is a Liouville domain, where $\lambda$ is the restriction to $X$ of the standard Liouville form
\eqref{eqn:standardLiouvilleform}.

\begin{definition}
\label{def:zetanondegenerate}
If $(X,\lambda)$ is a nondegenerate Liouville domain, define the {\bf zeta function\/}
\[
\zeta(X,\lambda) = \zeta(\partial X,\lambda|_{\partial X}).
\]
\end{definition}

\begin{definition}
Let $(X,\lambda)$ and $(X',\lambda')$ be Liouville domains. A symplectomorphism
\begin{equation}
\label{eqn:exact}
\phi: (\op{int}(X),d\lambda) \stackrel{\simeq}{\longrightarrow} (\op{int}(X'),d\lambda')
\end{equation}
is {\bf exact\/} if the closed 1-form $\phi^*\lambda'-\lambda$ on $\op{int}(X)$ is exact.
\end{definition}

We can now state the first theorem of this paper.

\begin{theorem}
\label{thm:nondegenerate}
(proved in \S\ref{sec:algebraic})
Let $(X,\lambda)$ and $(X',\lambda')$ be nondegenerate Liouville domains. Suppose there exists an exact symplectomorphism
\[
\phi: (\op{int}(X),d\lambda) \stackrel{\simeq}{\longrightarrow} (\op{int}(X'),d\lambda').
\]
Then
\[
\zeta(X,\lambda) = \zeta(X',\lambda').
\]
\end{theorem}

In particular, if $(X,\lambda)$ is a Liouville domain such that $H^1(X;\R)=0$, then the zeta function on the boundary of $X$ is determined by the symplectomorphism type of the interior of $X$.

\begin{remark}
(a) If there is an exact symplectomorphism of closed domains $(X,d\lambda)\stackrel{\simeq}{\longrightarrow} (X',d\lambda')$, then it follows directly that $\zeta(X,\lambda)=\zeta(X',\lambda')$, because the restriction of the symplectomorphism to the boundaries preserves the direction of the Reeb flow, and the symplectic actions of Reeb orbits are preserved. The nontrivial aspect of Theorem~\ref{thm:nondegenerate} is that we are starting with an exact symplectomorphism of open domains.

(b) Without the assumption that the symplectomorphism $\phi$ is exact, Theorem~\ref{thm:nondegenerate} is not true as stated, although it may be possible to prove a more complicated statement. For example, if the restriction map $H^1(X;\R)\to H^1(\partial X;\R)$ is nonzero, let $\alpha$ be a closed $1$-form on $X$ whose restriction to $\partial X$ is not exact. If $\epsilon>0$ is sufficiently small, then $\lambda' = \lambda + \epsilon\alpha$ is also a Liouville form on $X$. Then $\phi=\op{id}_X$ is a nonexact symplectomorphism $(X,d\lambda)\stackrel{\simeq}{\to}(X,d\lambda')$. The Reeb orbits of $\lambda$ and $\lambda'$ on $\partial X$ are the same up to reparametrization, but the symplectic actions are different for Reeb orbits in homology classes that pair nontrivially with $\alpha$.
\end{remark}

\begin{remark}
\label{rem:90s}
The relation of the interior of a compact symplectic domain to its boundary has been studied by various authors especially in the 1990s. (See also \cite{cgh} for recent results about open domains that ``do not have a well-defined symplectic boundary''.) In particular, results of Cieliebak-Floer-Hofer-Wysocki \cite{cfhw} using symplectic homology imply that under the hypotheses of Theorem~\ref{thm:nondegenerate}, the multiset of pairs $(\mathcal{A}(\gamma),\epsilon(\gamma))$, where $\gamma$ is a Reeb orbit of $\lambda$, is the same as the multiset of such pairs for $\lambda'$. Thus most of the ingredients in the zeta functions \eqref{eqn:defzeta} for $\lambda$ and $\lambda'$ are the same; however this approach does not detect the covering multiplicities $d(\gamma)$. Our approach will recover this last bit of information using {\em equivariant\/} symplectic homology.
\end{remark}


\subsection{Zeta functions of dynamically tame Liouville domains}

In \S\ref{sec:dynamicallytame} we will extend the definition of the zeta function from nondegenerate Liouville domains to a larger class of Liouville domains which we call ``dynamically tame''. Roughly speaking, a Liouville domain is dynamically tame if the Euler characteristic of the filtered equivariant symplectic homology changes only finitely many times below any given action level. For example, if the set of actions of Reeb orbits on the boundary is discrete, then the Liouville domain is dynamically tame; see Example~\ref{ex:discretespectrum}. The zeta function then encodes these changes in Euler characteristic. See \S\ref{sec:dynamicallytame} for the precise definitions.

Theorem~\ref{thm:nondegenerate} is now a special case of the following:

\begin{theorem}
\label{thm:tame}
(proved in \S\ref{sec:dynamicallytame})
Let $(X,\lambda)$ and $(X',\lambda')$ be dynamically tame Liouville domains. Suppose there exists an exact symplectomorphism
\[
\phi: (\op{int}(X),d\lambda) \stackrel{\simeq}{\longrightarrow} (\op{int}(X'),d\lambda').
\]
Then
\[
\zeta(X,\lambda) = \zeta(X',\lambda').
\]
\end{theorem}


\subsection{Examples in four dimensions}
\label{sec:introexamples}

The following are some computations of the zeta function for basic examples of four-dimensional Liouville domains. 

\begin{example}
Let $\Sigma$ be a closed surface, and let $g$ be a Riemannian metric on $\Sigma$. Let $D^*(\Sigma,g)$ denote the unit disk bundle in $T^*\Sigma$. Then there is a canonical Liouville form on $D^*(\Sigma,g)$, such that Reeb orbits on $\partial D^*(\Sigma, g)$ correspond to oriented closed geodesics on $\Sigma$, and symplectic action corresponds to Riemannian length. Suppose now that the metric $g$ has negative curvature (which is possible when $\Sigma$ has genus greater than $1$). Then the contact form on $\partial D^*(\Sigma,g)$ is nondegenerate, and every Reeb orbit is positive hyperbolic. It then follows from equation \eqref{eqn:3dimproduct} that
\begin{equation}
\label{eqn:geodesics}
\zeta(D^*(\Sigma,g)) = \prod_{\mbox{\scriptsize $\gamma$ unoriented simple closed geodesic}}\left(1-t^{\ell(\gamma)}\right)^2
\end{equation}
where $\ell(\gamma)$ denotes the length of $\gamma$. Here the factors in \eqref{eqn:geodesics} are squared because each unoriented simple closed geodesic $\gamma$ gives rise to two simple Reeb orbits.
\end{example}

We now consider the toric domains discussed in \S\ref{sec:motivating}.

\begin{proposition}
\label{prop:toric}
(proved in \S\ref{sec:computations})
Let $X_\Omega\subset\R^4$ be a star-shaped toric domain as in Definition~\ref{def:osstd}, such that $\partial_+\Omega$ is a curve from $(a,0)$ to $(0,b)$. Then $X_\Omega$ is dynamically tame, and
\begin{equation}
\label{eqn:zetatoric}
\zeta(X_\Omega) = \frac{1}{(1-t^a)(1-t^b)}.
\end{equation}
\end{proposition}

\begin{remark}
We expect that an analogue of Proposition~\ref{prop:toric} also holds for toric domains in higher dimensions.
\end{remark}

Equation \eqref{eqn:zetatoric} shows that the zeta function of a star-shaped toric domain has a very special form, and this allows us to detect many counterexamples for Question~\ref{question:motivating}. For example:

\begin{proof}[Proof of Proposition~\ref{prop:motivating}.]
Let $X\subset\R^4$ be a smooth star-shaped domain with a simple positive hyperbolic Reeb orbit $\gamma$ satisfying the hypotheses of the proposition. It follows from equation \eqref{eqn:3dimproduct} that $t^{\mathcal{A}(\gamma)}$ has coefficient $-1$ in $\zeta(X)$. On the other hand, if $X_\Omega\subset\R^4$ is a star-shaped toric domain, then by Proposition~\ref{prop:toric} and equation \eqref{eqn:novikovinverse}, all coefficients in $\zeta(X_\Omega)$ are nonnegative. It then follows from Theorem~\ref{thm:tame} that $\op{int}(X)$ cannot be symplectomorphic to $\op{int}(X_\Omega)$.  
\end{proof}

We can generalize this example to generic $S^1$-invariant star-shaped domains in $\R^4$. Let $f:\C P^1\to\R$ be a Morse function (not necessarily $C^2$-small), and let $X_f\subset\R^4$ be the smooth star-shaped domain defined by equation \eqref{eqn:Xf}. Let $\op{Crit}(f)$ denote the set of critical points of $f$. For $p\in\op{Crit}(f)$, let $\op{ind}(p)\in\{0,1,2\}$ denote the Morse index of $p$.

\begin{proposition}
\label{prop:s1inv}
(proved in \S\ref{sec:computations})
If $f:\C P^1\to\R$ is a Morse function, then $X_f$ is dynamically tame, and
\begin{equation}
\label{eqn:zetaXf}
\zeta(X_f) = \prod_{p\in\op{Crit}(f)}\left(1-t^{e^{f(p)}}\right)^{(-1)^{\op{ind}(p)-1}}.
\end{equation}
\end{proposition}

\begin{remark}
We suspect that the answer to the following question is no, although Proposition~\ref{prop:s1inv} does not seem sufficient to prove this:
Is every bounded open convex set in $\R^4$ with smooth boundary symplectomorphic to the interior of an $S^1$-invariant domain of the form $X_f$, where $f:\C P^1\to\R$ is a smooth function (not necessarily Morse)?
\end{remark}


\subsection{Outline of the rest of the paper}

Section~\ref{sec:prelim} reviews background on persistence modules (slightly modified for our situation) and defines a zeta function for persistence modules. Section \ref{sec:nondegenerate} reviews the filtered equivariant symplectic homology of nondegenerate Liouville domains. Section~\ref{sec:inverselimit} uses an inverse limit construction to extend the definition of filtered equivariant symplectic homology to all Liouville domains, and to show that it is invariant under exact symplectomorphisms of the interior. Section~\ref{sec:algebraic} uses some algebra to deduce invariance of the zeta function for open nondegenerate Liouville domains (Theorem~\ref{thm:nondegenerate}). Section~\ref{sec:degeneratecase} investigates the degenerate case in more detail, introduces the ``dynamically tame'' condition, and proves Theorem~\ref{thm:tame}. Finally, Section~\ref{sec:computations} computes examples of zeta functions (Propositions~\ref{prop:toric} and \ref{prop:s1inv}).

\begin{remark}
The definition of the zeta function in this paper does not require the full persistence module structure of equivariant symplectic homology, e.g.\ the lengths of bars; see the discussion in \S\ref{sec:zetabarcode}. In a followup paper, we plan to use more of the persistence module structure of equivariant symplectic homology to classify some open toric domains up to symplectomorphism.
\end{remark}

\paragraph{Acknowledgments.} The author thanks the participants of the conferences ``Persistence Homology in Symplectic and Contact Topology'' in Albi in June 2023, and ``Workshop on Conservative Dynamics and Symplectic Geometry'' in IMPA in August 2023, for helpful feedback on preliminary versions of this work. The author thanks the anonymous referee for many helpful comments.



\section{Algebraic preliminaries}
\label{sec:prelim}

\subsection{Persistence modules}
\label{sec:permod}

We now review what we will need to know about persistence modules.  The following are slight variants of standard definitions, as given for example in the book \cite{polterovich}.

\begin{notation}
If $V$ and $W$ are vector spaces over a field $\mathbb{F}$ with chosen bases, if $A:V\to W$ is a linear map, and if $x\in V$ and $y\in W$ are basis elements, then $\langle Ax,y\rangle\in\mathbb{F}$ denotes the coefficient of $y$ when $Ax$ is written as a linear combination of the basis elements of $W$.
\end{notation}

\begin{definition}
\label{def:pm}
A ($\Z/2$-graded, locally finite) {\bf persistence module\/} over a field $\mathbb{F}$ is a functor $\mathcal{H}$ from the ordered set $\R$ to the category of $\Z/2$-graded finite dimensional vector spaces over $\mathbb{F}$. Concretely, this consists of:
\begin{itemize}
\item
finite dimensional $\Z/2$-graded $\mathbb{F}$-vector spaces $\mathcal{H}^L=\mathcal{H}^L_0\oplus\mathcal{H}^L_1$ for each $L\in\R$,
\item
linear maps $\mathcal{H}^{L_2,L_1}:\mathcal{H}^{L_1}\to\mathcal{H}^{L_2}$ preserving the $\Z/2$-grading for each pair of real numbers $L_1,L_2$ with $L_1\le L_2$,
\end{itemize}
such that
\begin{itemize}
\item
$\mathcal{H}^{L,L}=\op{id}_{\mathcal{H}^L}$,
\item
$\mathcal{H}^{L_3,L_1}=\mathcal{H}^{L_3,L_2}\circ\mathcal{H}^{L_2,L_1}$ when $L_1\le L_2\le L_3$.
\end{itemize}
We further impose the following two conditions:
\begin{itemize}
\item (bounded from below)
$\mathcal{H}^{L}=\{0\}$ for $L\ll0$.
\item (proper)
There is a closed and discrete set $A\subset\R$ such that if $L_1\le L_2$ and if the interval $(L_1,L_2]\subset\R\setminus A$, then the map $\mathcal{H}^{L_2,L_1}$ is an isomorphism.
\end{itemize} 
\end{definition}

A basic example of a persistence module arises as follows.

\begin{definition}
\label{def:fcc}
A ($\Z/2$-graded, locally finite) {\bf $\R$-filtered chain complex\/} over a field $\mathbb{F}$ is 
a chain complex $(C,\partial)$ over $\mathbb{F}$ with a distinguished basis $\mathcal{B}$, a $\Z/2$-grading $\epsilon:\mathcal{B}\to\Z/2$, and a ``filtration'' $\mathcal{F}:\mathcal{B}\to\R$, with the following properties:
\begin{itemize}
\item
If $\langle \partial x,y\rangle \neq 0$, then $\epsilon(x)\neq \epsilon(y)$, and $\mathcal{F}(x) > \mathcal{F}(y)$.
\item
For each $L\in\R$, there are only finitely many $x\in\mathcal{B}$ with $\mathcal{F}(x)\le L$.
\end{itemize}
\end{definition}

\begin{notation}
\label{not:fh}
Let $(C,\partial,\mathcal{B},\epsilon,\mathcal{F})$ be an $\R$-filtered chain complex over $\mathbb{F}$ as in Definition~\ref{def:fcc}. For each $L\in\R$, there is a subcomplex $C^{\mathcal{F}\le L}$ given by the span of the basis elements $x\in\mathcal{B}$ with $\mathcal{F}(x)\le L$. Let $H_*^{\mathcal{F}\le L}(C,\partial)$ denote the homology of this subcomplex. For $L_1\le L_2$, let
\begin{equation}
\label{eqn:iim}
\imath^{L_2,L_1}: H_*^{\mathcal{F}\le L_1}(C,\partial) \longrightarrow H_*^{\mathcal{F}\le L_2}(C,\partial)
\end{equation}
denote the map on homology induced by the inclusion of chain complexes.
\end{notation}

\begin{example}
\label{ex:fcc}
Let $(C,\partial,\mathcal{B},\epsilon,\mathcal{F})$ be an $\R$-filtered chain complex over $\mathbb{F}$ as in Definition~\ref{def:fcc}. We can then define a persistence module as in Definition~\ref{def:pm} by setting $\mathcal{H}^L= H_*^{\mathcal{F}\le L}(C,\partial)$, with the $\Z/2$-grading given by $\epsilon$, and $\mathcal{H}^{L_2,L_1} = \imath^{L_2,L_1}$.  Here the properness condition in Definition~\ref{def:pm} is fulfilled by $A=\{\mathcal{F}(x)\mid x\in\mathcal{B}\}$.
\end{example}

\begin{definition}
\label{def:morphism}
If $\mathcal{G}$ is another persistence module as in Definition~\ref{def:pm}, a {\bf morphism\/} $\phi:\mathcal{G}\to\mathcal{H}$ is a natural transformation of functors. Concretely, this consists of linear maps $\phi^L:\mathcal{G}^L\to \mathcal{H}^L$ preserving the $\Z/2$ grading for each $L\in\R$, such that if $L_1\le L_2$ then
\[
\mathcal{H}^{L_2,L_1}\circ\phi^{L_1} = \phi^{L_2}\circ \mathcal{G}^{L_2,L_1}: \mathcal{G}^{L_1} \longrightarrow \mathcal{H}^{L_2}.
\]
\end{definition}

\begin{example}
Continuing with Example~\ref{ex:fcc}, suppose that $(C',\partial',\mathcal{B}',\epsilon',\mathcal{F}')$ is another $\R$-filtered chain complex as in Definition~\ref{def:fcc}, and let $\mathcal{G}$ denote the associated persistence module. Let $\phi:C'\to C$ be a chain map such that if $x\in\mathcal{B}'$, if $y\in\mathcal{B}$, and if $\langle\phi x,y\rangle\neq 0$, then $\epsilon'(x)=\epsilon(y)$ and $\mathcal{F}'(x)\ge \mathcal{F}(y)$.  Then the induced maps on homology $\phi^L:\mathcal{G}^L\to \mathcal{H}^L$ constitute a morphism as in Definition~\ref{def:morphism}.
\end{example}

\begin{notation}
Let $\op{PerMod}_{\mathbb F}$ denote the set of persistence modules over $\mathbb{F}$ as in Definition~\ref{def:pm}, modulo isomorphisms (i.e.\ invertible morphisms as in Definition~\ref{def:morphism}).
\end{notation}

\begin{definition}
\label{def:shift}
If $\mathcal{H}$ is a persistence module as in Definition~\ref{def:pm}, and if $\delta\in\R$, the {\bf shift\/} of $\mathcal{H}$ by $\delta$ is the persistence module $\mathcal{H}[\delta]$ defined as follows. For $L\in\R$ we set $\mathcal{H}[\delta]^L=\mathcal{H}^{L+\delta}$, and for $L_1\le L_2$ we set
\[
\mathcal{H}[\delta]^{L_2,L_1}=\mathcal{H}^{L_2+\delta,L_1+\delta}: \mathcal{H}^{L_1+\delta} \longrightarrow \mathcal{H}^{L_2+\delta}.
\]
If $\delta\ge 0$, we define the {\bf shift morphism\/}
\begin{equation}
\label{eqn:shiftmorphism}
\eta_\delta: \mathcal{H} \longrightarrow \mathcal{H}[\delta]
\end{equation}
to be the morphism of persistence modules given by
\[
\eta_\delta^L = \mathcal{H}^{L+\delta,L}: \mathcal{H}^L \longrightarrow \mathcal{H}^{L+\delta} = \mathcal{H}[\delta]^L.
\]
\end{definition}


\subsection{The zeta function of a persistence module}
\label{sec:zetapm}

We now define a zeta function for persistence modules, in preparation for discussing zeta functions of Liouville domains.

Let $\mathcal{H}$ be a persistence module. Recall from Definition~\ref{def:pm} that there is a closed and discrete set $A\subset \R$ such that if $L_1\le L_2$ and if the interval $(L_1,L_2]\subset\R\setminus A$, then the map $\mathcal{H}^{L_2,L_1}$ is an isomorphism. This implies that for every $L\in\R$, the $\Z/2$-graded vector spaces $\mathcal{H}^{L-\delta}$ for $\delta>0$ sufficiently small are canonically isomorphic, via the maps $\mathcal{H}^{L-\delta_1,L-\delta_2}$ for $\delta_1\le \delta_2$, to a single $\Z/2$-graded vector space, which we denote by $\mathcal{H}^{L,-}$. Similarly, the $\Z/2$-graded vector spaces $\mathcal{H}^{L+\delta}$ for $\delta\ge 0$ sufficiently small are canonically isomorphic to $\mathcal{H}^{L}$.

\begin{definition}
\label{def:DeltaL}
If $\mathcal{H}$ is a persistence module and if $L\in\R$, then using the notation above, define the {\bf Euler characteristic jump\/}
\begin{equation}
\label{eqn:DeltaL}
\Delta_L(\mathcal{H}) = \left(\dim(\mathcal{H}^{L}_0) - \dim(\mathcal{H}^{L}_1)\right) - \left(\dim(\mathcal{H}^{L,-}_0) - \dim(\mathcal{H}^{L,-}_1)\right) \in\Z.
\end{equation}
\end{definition}

Note that if $L\notin A$, then $\mathcal{H}_\epsilon^{L}\simeq \mathcal{H}_\epsilon^{L,-}$ so $\Delta_L(\mathcal{H})=0$.

\begin{definition}
\label{def:zetapm}
We define a function
\[
\zeta : \op{PerMod}_{\mathbb{F}}\longrightarrow \Lambda
\]
by
\begin{equation}
\label{eqn:zetapm}
\zeta(\mathcal{H}) = \sum_{L\in\R}\Delta_L(\mathcal{H})t^L.
\end{equation}
\end{definition}

Note that equation \eqref{eqn:zetapm} is invariant under isomorphisms of persistence modules, by equation \eqref{eqn:DeltaL}. 

\subsection{Digression: zeta functions in terms of barcodes}
\label{sec:zetabarcode}

A fundamental result is that persistence modules modulo isomorphism are equivalent to ``barcodes''. We now describe how this works for our version of ``persistence module'' and describe the zeta function in terms of barcodes. The material in this subsection is included for additional context and is not needed in the rest of the paper.

\begin{definition}
A ($\Z/2$-graded) {\bf bar\/} is a pair $(I,\epsilon)$ where $I$ is an interval of the form $[a,b)$ with $a,b\in\R$ and $a<b$ or $[a,\infty)$ with $a\in\R$, and $\epsilon\in\Z/2$. Let $\op{Bar}$ denote the set of bars.
\end{definition}

\begin{definition}
\label{def:barcode}
A (locally finite, proper, $\Z/2$-graded) {\bf barcode\/} is a function
\[
f:\op{Bar}\longrightarrow \Z^{\ge 0}
\]
such that:
\begin{description}
\item{(*)}
For each $L\in\R$, there are only finitely many bars $(I,\epsilon)\in\op{Bar}$ such that $f(I,\epsilon)>0$ and the lower endpoint of the interval $I$ is less than or equal to $L$.
\end{description}
We let $\op{Barcodes}$ denote the set of barcodes.
\end{definition}

We can regard a barcode $f$ as a multiset of bars $(I,\epsilon)\in\op{Bar}$, where $f(I,\epsilon)$ is the number of times that the bar $(I,\epsilon)$ appears in the multiset.

\begin{definition}
Fix a field $\mathbb{F}$. Given a bar $(I,\epsilon)\in\op{Bar}$, define a persistence module $\mathcal{H}=\mathcal{H}_{\mathbb F}(I,\epsilon)$ over $\mathbb{F}$ as follows.
\begin{itemize}
\item
If $\epsilon'\neq\epsilon$ then $\mathcal{H}_{\epsilon'}^L=\{0\}$ for all $L\in\R$.
\item
We set
\[
\mathcal{H}_\epsilon^L=\left\{\begin{array}{cl} \mathbb{F}, & L\in I,\\
0, & L\notin I.
\end{array}\right.
\]
\item
For $L_1\le L_2$, the map $\mathcal{H}^{L_2,L_1}$ is the identity on $\mathbb{F}$ when $L_1,L_2\in I$, and $0$ otherwise.
\end{itemize}
\end{definition}

\begin{definition}
We define a map
\begin{equation}
\label{eqn:normalform}
\Phi_{\mathbb F}: \op{Barcodes} \longrightarrow \op{PerMod}_{\mathbb F}
\end{equation}
as follows. If $f:\op{Bar}\to\Z^{\ge 0}$ is a barcode, then
\[
\Phi_{\mathbb F}(f) = \bigoplus_{(I,\epsilon)\in\op{Bar}} \bigoplus_{f(I,\epsilon)}\mathcal{H}_{\mathbb F}(I,\epsilon).
\]
\end{definition}

Note that $\Phi_{\mathbb F}(f)$ satisfies the boundedness and properness requirements in Definition~\ref{def:pm} as a consequence of condition (*) in Definition~\ref{def:barcode}.

We now have the following fundamental result. 

\begin{theorem}
\label{thm:fundamental}
The map $\Phi_{\mathbb F}$ is a bijection.
\end{theorem}

\begin{proof}
This is a slight variant of the ``Normal Form Theorem'' in \cite[\S2]{polterovich}.
The proof from \cite{polterovich} carries over directly.
\end{proof}

One can also show, as a slight variant of the ``Isometry Theorem'' in \cite[\S2]{polterovich}, that the map $\Phi_{\mathbb F}$ is an isometry with respect to the ``bottleneck distance'' on barcodes and the ``interleaving distance'' on persistence modules.

One can now define the zeta function of a barcode as follows. In short, {\em the zeta function records the endpoints of the bars, but forgets which endpoints are connected to each other by bars.\/}

\begin{definition}
\label{def:zetabar}
We define a map
\[
\zeta: \op{Barcodes} \longrightarrow \Lambda
\]
as follows. If $f:\op{Bar}\to\Z^{\ge 0}$ is a barcode, then
\[
\zeta(f) = \sum_{(I,\epsilon)\in\op{Bar}}
f(I,\epsilon)
\left\{\begin{array}{cl} (-1)^\epsilon t^a, & I=[a,\infty),\\
(-1)^\epsilon(t^a-t^b), & I=[a,b).
\end{array}
\right.
\]
\end{definition}

Note that condition (*) in Definition~\ref{def:barcode} implies that $\zeta(f)$ is a well defined element of the Novikov ring $\Lambda$. Note also that Definitions~\ref{def:zetabar} and \ref{def:zetapm} agree under the bijection \eqref{eqn:normalform}.

\begin{remark}
For finitely generated persistence modules, a similar notion of ``persistent magnitude'' is defined in \cite[Def.\ 5.4]{govc}. In our notation, the persistent magnitude is obtained by substituting $t=e^{-1}$ into the zeta function and then multiplying by $-1$. Some other related notions are defined in \cite[\S5.1]{bcz}.
\end{remark}



\section{Filtered equivariant symplectic homology of nondegenerate Liouville domains}
\label{sec:nondegenerate}

Let $(X,\lambda)$ be a nondegenerate Liouville domain, see Definition~\ref{def:nld}, and let $L\in \R$. Given these data, one can define the {\bf filtered positive $S^1$-equivariant symplectic homology\/}, which we denote\footnote{
The motivation for the notation $CH$ is that, as explained in \cite{bo}, positive $S^1$-equivariant symplectic homology is isomorphic to linearized contact homology (see e.g.\ \cite[\S3.1]{bee}), when sufficient transversality for holomorphic curves holds to define the latter. For easier comprehension one can pretend that $CH$ is linearized contact homology, with the understanding that to give rigorous definitions and proofs without transversality difficulties we use positive $S^1$-equivariant symplectic homology instead.
} by $CH^L(X,\lambda)$.
The precise definition of $CH^L(X,\lambda)$ that we are using\footnote{A notational difference between here and \cite{gh} is that $CH^L(X,\lambda)$ here corresponds to $CH^{e^L}(X,\lambda)$ in \cite{gh}. The idea is that \cite{gh} considers Reeb orbits with action less than or equal to $L$, while we consider Reeb orbits with action less than or equal to $e^L$. The reason for the convention here is that we want scaling of the contact form to correspond to a shift morphism \eqref{eqn:shiftmorphism} on the persistence module coming from equivariant symplectic homology.} is given in \cite{gh}, and further properties are established in \cite{gu}. (Positive) $S^1$-equivariant symplectic homology was originally defined by Viterbo \cite{viterbo}, extending the definition of nonequivariant symplectic homology by Cieliebak, Floer, and Hofer \cite{fh,cfh}. An alternate version of (positive) $S^1$-equivariant symplectic homology was defined using family Floer homology by Bourgeois-Oancea \cite[\S2.2]{bo}, following a suggestion of Seidel \cite{seidel}, and further developed by Gutt \cite{gutt}. The filtered version used here is based on this family Floer approach.

\subsection{Basic properties}

We now review the properties of $CH^L$ that we will need, which are collected in Proposition~\ref{prop:CHL} below.

\begin{definition}
\label{def:good}
Let $(Y,\lambda)$ be a nondegenerate contact manifold. A Reeb orbit $\gamma$ is {\bf bad\/} if it is the $d$-fold cover of a simple Reeb orbit $\overline{\gamma}$ where $d$ is even and the Lefschetz signs $(-1)^{\epsilon(\gamma)}\neq(-1)^{\epsilon(\overline{\gamma})}$. A Reeb orbit is {\bf good\/} if it is not bad.
\end{definition}

\begin{remark}
\label{rem:czl}
The original definition of good and bad Reeb orbits in \cite{egh} is given in terms of the mod 2 Conley-Zehnder index, which we denote here by $\op{CZ}$, instead of the Lefschetz sign. This is equivalent, because we have
\[
\op{CZ}(\gamma)\equiv \epsilon(\gamma)+n-1\mod 2,
\]
where $\dim(Y)=2n-1$.
\end{remark}

\begin{definition}
Let $(X,\lambda)$ and $(X',\lambda')$ be Liouville domains of the same dimension. An {\bf exact symplectic embedding\/} $(X,\lambda)\hookrightarrow (X',\lambda')$ is a symplectic embedding $\varphi:(X,d\lambda)\hookrightarrow (X',d\lambda')$ such that
\[
\left[\varphi^*\lambda'-\lambda\right] = 0 \in H^1(X;\R).
\]
\end{definition}

\begin{example}
\label{ex:inclusion}
If $(X',\lambda')$ is a Liouville domain, and if $X\subset X'$ is a subset such that $(X,\lambda'|_X)$ is a Liouville domain of the same dimension, then the inclusion map $(X,\lambda'|_X)\hookrightarrow (X',\lambda')$ is an exact symplectic embedding.
\end{example}

\begin{example}
\label{ex:scaling}
Let $(X,\lambda)$ be a Liouville domain. There is a unique vector field $V$ on $X$ such that $\imath_Vd\lambda = \lambda$. This vector field is transverse to $\partial X$ and points outward from $X$. Given $\delta\ge 0$, let $\Phi_X^\delta:X\to X$ denote the time $-\delta$ flow of $V$. Then $(\Phi_X^\delta)^*\lambda = e^{-\delta}\lambda$. Thus we have an exact symplectic embedding
\[
\Phi_X^\delta: (X,e^{-\delta}\lambda) \hookrightarrow (X,\lambda).
\]
\end{example}

\begin{remark}
It follows from the definition that if $\varphi:(X,\lambda)\hookrightarrow (X',\lambda')$ and $\varphi':(X',\lambda')\hookrightarrow (X'',\lambda'')$ are exact symplectic embeddings, then the composition
\[
\varphi'\circ\varphi: (X,\lambda)\hookrightarrow (X'',\lambda'')
\]
is also an exact symplectic embedding.
\end{remark}

\begin{proposition}
\label{prop:CHL}
Let $(X,\lambda)$ be a nondegenerate Liouville domain and let $L\in\R$. Then the filtered positive $S^1$-equivariant symplectic homology $CH^L(X,\lambda)$ is a $\Z/2$-graded\footnote{In some cases, for example when $X$ is a smooth star-shaped domain in $\R^{2n}$, the $\Z/2$-grading has a refinement to a $\Z$-grading determined by the Conley-Zehnder index. However we will not need this.} $\Q$-module with the following properties:
\begin{description}
\item{(Persistence)} If $L_1\le L_2$, then there is a canonical ``persistence morphism''
\begin{equation}
\label{eqn:persistence}
CH^{L_2,L_1}(X,\lambda): CH^{L_1}(X,\lambda) \longrightarrow CH^{L_2}(X,\lambda)
\end{equation}
such that:
\begin{itemize}
\item
$CH^{L,L}(X,\lambda) = \op{id}_{CH^L(X,\lambda)}$.
\item
If $L_1\le L_2 \le L_3$, then
\[
CH^{L_3,L_1}(X,\lambda) = CH^{L_3,L_2}(X,\lambda) \circ CH^{L_2,L_1}(X,\lambda).
\]
\end{itemize}
\item{(Reeb Orbits)}
There exists\footnote{The differential on the chain complex is not uniquely determined by $(X,\lambda)$ and $L$ and depends on some additional choices.} an $\R$-filtered chain complex $(C,\partial,\mathcal{B},\epsilon,\mathcal{F})$ over $\Q$ as in Definition~\ref{def:fcc} with the following properties:
\begin{itemize}
\item $\mathcal{B}$ is the set of good Reeb orbits in $(\partial X,\lambda)$.
\item If $\gamma$ is a good Reeb orbit, then $(-1)^{\epsilon(\gamma)}$ is the Lefschetz sign of $\gamma$ as in \S\ref{sec:zetanondeg}, and $\mathcal{F}(\gamma) = \log\mathcal{A}(\gamma)$, where $\mathcal{A}$ is the symplectic action as in \S\ref{sec:zetanondeg}.
\item 
In the notation of \S\ref{sec:permod}, there is an isomorphism
\begin{equation}
\label{eqn:awkward}
CH^L(X,\lambda) \simeq H_*^{\mathcal{F}\le L}(C,\partial).
\end{equation}
\item 
If $L_1\le L_2$, then the image of $CH^{L_2,L_1}(X,\lambda)$ is isomorphic under \eqref{eqn:awkward} (with $L=L_2$) to the image of the inclusion-induced map \eqref{eqn:iim}.
\end{itemize}
\item{(Transfer Morphisms)}
If $(X',\lambda')$ is another nondegenerate Liouville domain, and if $\varphi:(X,\lambda) \hookrightarrow (\op{int}(X'),\lambda')$ is an exact symplectic embedding\footnote{Instead of an exact symplectic embedding, one can start with a slight generalization called a ``generalized Liouville embedding'' in \cite[Def.\ 1.23]{gh}. However we will not need this.}, then there is a well-defined {\bf transfer morphism\/}
\begin{equation}
\label{eqn:transfer}
CH^L(\varphi): CH^L(X',\lambda') \longrightarrow CH^L(X,\lambda)
\end{equation}
with the following properties:
\begin{itemize}
\item
If $L_1\le L_2$, then we have a commutative diagram
\begin{equation}
\label{eqn:transferfiltration}
\begin{CD}
CH^{L_1}(X',\lambda') @>{CH^{L_1}(\varphi)}>> CH^{L_1}(X,\lambda)\\
@V{CH^{L_2,L_1}(X',\lambda')}VV  @VV{CH^{L_2,L_1}(X,\lambda)}V\\
CH^{L_2}(X',\lambda') @>CH^{L_2}(\varphi)>> CH^{L_2}(X,\lambda).
\end{CD}
\end{equation}
\item
If $(X'',\lambda'')$ is a third nondegenerate Liouville domain, and if $\varphi':(X',\lambda')\hookrightarrow (\op{int}(X''),\lambda'')$ is another exact symplectic embedding, then we have
\begin{equation}
\label{eqn:transferfunctor}
CH^L(\varphi'\circ\varphi) = CH^L(\varphi)\circ CH^L(\varphi') : CH^L(X'',\lambda'') \longrightarrow CH^L(X,\lambda).
\end{equation}
\end{itemize}
\item{(Scaling)}
If $\delta\in\R$, then there is a canonical {\bf scaling isomorphism}
\[
s_\delta : CH^L(X,\lambda) \stackrel{\simeq}{\longrightarrow} CH^{L-\delta}(X,e^{-\delta}\lambda).
\]
with the following properties:
\begin{itemize}
\item
The scaling isomorphism commutes with the persistence maps \eqref{eqn:persistence}. That is, if $L_1\le L_2$, then we have a commutative diagram
\begin{equation}
\label{eqn:scaling1}
\begin{CD}
CH^{L_1}(X,\lambda) @>{s_\delta}>{\simeq}> CH^{L_1-\delta}(X,e^{-\delta}\lambda)\\
@V{CH^{L_2,L_1}(X,\lambda)}VV @VV{CH^{L_2-\delta,L_1-\delta}(X,e^{-\delta}\lambda)}V\\
CH^{L_2}(X,\lambda) @>{s_\delta}>{\simeq}> CH^{L_2-\delta}(X,e^{-\delta}\lambda).
\end{CD}
\end{equation}
\item
If $\delta> 0$, and if $\Phi_X^\delta:(X,e^{-\delta}\lambda)\to (X,\lambda)$ is the exact symplectic embedding in Example~\ref{ex:scaling}, then we have
\begin{equation}
\label{eqn:scaling3}
CH^L(\Phi_X^\delta) = CH^{L,L-\delta}(X,e^{-\delta}\lambda) \circ s_\delta : CH^L(X,\lambda) \longrightarrow CH^L(X,e^{-\delta}\lambda).
\end{equation}
\end{itemize}
\end{description}
\end{proposition}

\begin{proof}

{\em Definition of $CH^L(X,\lambda)$.}
We define $CH^L(X,\lambda)$ to be the $\Q$-module that is called $CH^{e^L}(X,\lambda)$ in \cite{gh}. 

In the special case when $X$ is a star-shaped domain in $\R^{2n}$, this module has a $\Z$-grading explained in \cite[\S6.7]{gh} in terms of the Conley-Zehnder index. For a general nondegenerate Liouville domain, the Conley-Zehnder index is still well defined mod 2, and so the convention in \cite[\S6.7]{gh} extends to define a $\Z/2$-grading on $CH^L(X,\lambda)$. In light of Remark~\ref{rem:czl}, we further shift this grading by $n-1\mod 2$, where $\dim(X)=2n$.

{\em Persistence.} This is part of the ``Action filtration'' property in \cite[Prop.\ 3.1]{gh}.

{\em Reeb Orbits.} This follows from \cite[Lem.\ 2.1]{gu} in the case when $c_1(TX)|_{\pi_2(X)}=0$. The latter hypothesis was made in \cite{gu} in order to define a $\Z$-grading on $CH^L(X,\lambda)$; without this hypothesis, the proof carries over using the $\Z/2$-grading above.

{\em Transfer morphisms.} The construction of the transfer morphism \eqref{eqn:transfer} and the proof of the commutativity of the diagram \eqref{eqn:transferfiltration} are given in the proof of the ``Action'' property in \cite[Prop.\ 3.3]{gh}. The functoriality \eqref{eqn:transferfunctor} for transfer morphisms on unfiltered positive $S^1$-equivariant symplectic homology for Liouville embeddings is proved in \cite[Thm.\ 4.12]{gutt}. This extends to the filtered positive $S^1$-equivariant homology $CH^L$ by the arguments in \cite[\S8.1]{gh}, and to exact symplectic embeddings by the construction in \cite[\S7.3]{gh}.

{\em Scaling.} The construction of the scaling isomorphism and the proof of the first bullet point \eqref{eqn:scaling1} are given in the proof of the ``Scaling'' property in \cite[Prop.\ 3.1]{gh}. The remaining bullet point \eqref{eqn:scaling3} follows directly from the definition of the scaling isomorphism in \cite[\S6.6]{gh}.
\end{proof}

\begin{remark}
\label{rem:awkward}
The Reeb orbits property does not assert that the isomorphism \eqref{eqn:awkward} commutes with the maps $CH^{L_2,L_1}(X,\lambda)$ and \eqref{eqn:iim}, although one might expect this from the linearized contact homology perspective. The cause of this awkwardness is that  $CH^L(X,\lambda)$ is more naturally described as the homology of a chain complex with infinitely many generators for every Reeb orbit (good or bad) with symplectic action $\le e^L$; see \cite[Rem.\ 5.15]{gh}. The reduction to a chain complex with one generator for every good Reeb orbit with action $\le e^L$ in the Reeb Orbits property of Proposition~\ref{prop:CHL} uses a spectral sequence argument.
\end{remark}


\subsection{Equivariant symplectic homology as a persistence module}

We now recast Proposition~\ref{prop:CHL} using the language of persistence modules, as Proposition~\ref{prop:CH} below. Some related discussion in the context of nonequivariant symplectic homology is given in \cite[\S3]{sz} and \cite[\S9]{polterovich}.

\begin{definition}
Let $(Y,\lambda)$ be a contact manifold. Let $\mathcal{P}(Y,\lambda)$ denote the set of Reeb orbits. Define the {\bf action spectrum\/}
\[
\op{Spec}(Y,\lambda) = \{\mathcal{A}(\gamma) \mid \gamma\in\mathcal{P}(Y,\lambda)\}.
\]
\end{definition}

\begin{lemma}
\label{lem:intervaliso}
Let $(X,\lambda)$ be a nondegenerate Liouville domain, let $L_1<L_2$, and suppose that
\begin{equation}
\label{eqn:nospectrumininterval}
(e^{L_1},e^{L_2}] \cap \op{Spec}(\partial X,\lambda) = \emptyset.
\end{equation}
Then the map
\[
CH^{L_2,L_1}(X,\lambda) : CH^{L_1}(X,\lambda) \longrightarrow CH^{L_2}(X,\lambda)
\]
is an isomorphism.
\end{lemma}

\begin{proof}
This follows from the ``Reeb orbits'' property in \cite[Prop.\ 3.1]{gh}. We can also deduce this from Proposition~\ref{prop:CHL} as follows.

Let $(C,\partial,\mathcal{B},\epsilon,\mathcal{F})$ be an $\R$-filtered chain complex given by the Reeb Orbits property in Proposition~\ref{prop:CHL}. Consider the diagram
\[
\begin{CD}
H_*^{\mathcal{F}\le L_1}(C,\partial) @>{\imath^{L_2,L_1}}>> H_*^{\mathcal{F}\le L_2}(C,\partial) \\
@V{\simeq}VV  @V{\simeq}VV\\
CH^{L_1}(X,\lambda) @>{CH^{L_2,L_1}(X,\lambda)}>> CH^{L_2}(X,\lambda).
\end{CD}
\]
Here the vertical isomorphisms are given by \eqref{eqn:awkward}. As noted in Remark~\ref{rem:awkward}, this diagram might not commute. However we do know by the hypothesis \eqref{eqn:nospectrumininterval} that the top horizontal arrow is an isomorphism. It then follows from the last bullet point in the Reeb Orbits property in Proposition~\ref{prop:CHL} that the bottom horizontal arrow is surjective. Since this is a linear map between finite dimensional vector spaces of the same dimension, it must then be an isomorphism.
\end{proof}

\begin{definition}
\label{def:plgood}
Suppose that $(Y,\lambda)$ is a nondegenerate contact manifold. Let $\mathcal{P}_{\op{good}}(Y,\lambda)$ denote the set of good Reeb orbits. If $L\in\R$, define
\begin{equation}
\label{eqn:Pgood}
\#\mathcal{P}_{\op{good}}^L(Y,\lambda) = \sum_{\gamma\in\mathcal{P}_{\op{good}}(Y,\lambda) : \mathcal{A}(\gamma) = e^L} (-1)^{\epsilon(\gamma)} \in \Z.
\end{equation}
\end{definition}

\begin{proposition}
\label{prop:CH}
Let $(X,\lambda)$ be a nondegenerate Liouville domain. Then:
\begin{description}
\item{(Persistence)}
The $\Z/2$-graded $\Q$-modules $CH^L(X,\lambda)$ and the maps $CH^{L_2,L_1}(X,\lambda)$ in Proposition~\ref{prop:CHL} constitute a persistence module $CH(X,\lambda)$ over $\Q$ as in Definition~\ref{def:pm}.
\item{(Reeb Orbits)}
If $L\in\R$, then the Euler characteristic jump in Definition~\ref{def:DeltaL} is given by
\begin{equation}
\label{eqn:ecj}
\Delta_L(CH(X,\lambda)) = \#\mathcal{P}_{\op{good}}^{L}(\partial X,\lambda).
\end{equation}
\item{(Transfer Morphisms)}
 If $(X',\lambda')$ is another nondegenerate Liouville domain and if $\varphi:(X,\lambda) \hookrightarrow (\op{int}(X'),\lambda')$ is an exact symplectic embedding, then the transfer morphisms $CH^L(\varphi)$ in Proposition~\ref{prop:CHL} constitute a morphism of persistence modules
\[
CH(\varphi) : CH(X',\lambda') \longrightarrow CH(X,\lambda).
\]
If $(X'',\lambda'')$ is a third nondegenerate Liouville domain, and if $\varphi':(X',\lambda')\hookrightarrow (\op{int}(X''),\lambda'')$ is another exact symplectic embedding, then we have
\[
CH(\varphi'\circ\varphi) = CH(\varphi)\circ CH(\varphi') : CH(X'',\lambda'') \longrightarrow CH(X,\lambda).
\]
\item{(Scaling)}
If $\delta\in\R$, then in the notation of Definition~\ref{def:shift}, there is a canonical isomorphism of persistence modules
\[
s_\delta : CH(X,\lambda) \stackrel{\simeq}{\longrightarrow} CH(X,e^{-\delta}\lambda)[-\delta].
\]
If $\delta> 0$, and if $\Phi_X^\delta:(X,e^{-\delta}\lambda)\hookrightarrow (X,\lambda)$ is the exact symplectic embedding in Example~\ref{ex:scaling}, then we have
\[
CH(\Phi_X^\delta) = \eta_\delta \circ s_\delta : CH(X,\lambda) \longrightarrow CH(X,e^{-\delta}\lambda),
\]
where $\eta_\delta$ is the shift morphism \eqref{eqn:shiftmorphism}.
\end{description}
\end{proposition}

\begin{proof}
{\em Persistence.}
Since $(\partial X,\lambda)$ is compact and nondegenerate, for each $L\in\R$ there are only finitely many Reeb orbits $\gamma$ with ${\mathcal{A}(\gamma)} \le e^L$. Together with the Reeb Orbits property in Proposition~\ref{prop:CHL}, this implies the finite dimensionality and boundedness from below in Definition~\ref{def:pm}. It follows from Lemma~\ref{lem:intervaliso} that the properness condition in Definition~\ref{def:pm} holds with $A=\{L\mid e^L\in \op{Spec}(\partial X,\lambda)\}$. The remaining conditions in Definition~\ref{def:pm} follow from the Persistence property in Proposition~\ref{prop:CHL}.

{\em Reeb Orbits.}
If $e^L \notin\op{Spec}(\partial X,\lambda)$, then it follows from equation \eqref{eqn:DeltaL} and the isomorphism \eqref{eqn:awkward} that $\Delta_L(CH(X,\lambda))=0$, and it follows from equation \eqref{eqn:Pgood} that $\#\mathcal{P}_{\op{good}}^L(\partial X,\lambda) = 0$.

It remains to prove equation \eqref{eqn:ecj} when $e^L\in\op{Spec}(\partial X,\lambda)$. Choose $\delta>0$ sufficiently small so that there are no Reeb orbits with symplectic action in the interval $(e^{L-\delta},e^L)$. Let $V=V_0\oplus V_1$ be the $\Z/2$-graded vector space over $\Q$ with basis given by the good Reeb orbits $\gamma$ with action $\mathcal{A}(\gamma)=e^L$, where the $\Z/2$ grading of a Reeb orbit $\gamma$ is given by $\epsilon(\gamma)$, where $(-1)^{\epsilon(\gamma)}$ is the Lefschetz sign. By Definition~\ref{def:plgood} we have
\[
\#\mathcal{P}^{L}_{\op{good}}(\partial X,\lambda) = \dim(V_0) - \dim(V_1).
\]
Thus we need to show that
\begin{gather}
\nonumber
\dim(CH^{L}_0(X,\lambda)) - \dim(CH^{L}_1(X,\lambda)) - \dim(CH^{L-\delta}_0(X,\lambda)) + \dim(CH^{L-\delta}_1(X,\lambda))\\
\label{eqn:sixterm}
= \dim(V_0) - \dim(V_1).
\end{gather}

Let $(C,\partial,\mathcal{B},\epsilon,\mathcal{F})$ be an $\R$-filtered chain complex provided by the Reeb Orbits property in Proposition~\ref{prop:CHL}. We then have a short exact sequence of $\Z/2$-graded chain complexes
\[
0 \longrightarrow C^{\mathcal{F}\le L-\delta} \longrightarrow C^{\mathcal{F}\le L} \longrightarrow V \longrightarrow 0
\]
where the differential on $V$ is zero. This gives a six-term exact sequence on homology, which by the isomorphism \eqref{eqn:awkward} can be identified\footnote{The map $CH^{L-\delta}(X,\lambda)\to CH^{L}(X,\lambda)$ in this exact sequence corresponds to the map $\imath^{L,L-\delta}:H_*^{\mathcal{F}\le L-\delta}(C,\partial) \to H_*^{\mathcal{F}\le L}(C,\partial)$. Note that this map is not necessarily identified with the map $CH^{L,L-\delta}(X,\lambda)$; see Remark~\ref{rem:awkward}.
}
with
\[
\cdots\to
CH_0^{L-\delta}(X,\lambda) \to CH_0^{L}(X,\lambda) \to V_0
\to CH_1^{L-\delta}(X,\lambda) \to CH_1^{L}(X,\lambda) \to V_1 \to \cdots
\]
The alternating sum of the dimensions of the six terms is zero, and this gives equation \eqref{eqn:sixterm}.

{\em Transfer Morphisms and Scaling:} These properties are rewrites of the corresponding properties in Proposition~\ref{prop:CHL}, using the definitions in \S\ref{sec:permod}.
\end{proof}

\begin{remark}
In the language of Proposition~\ref{prop:CH}, Remark~\ref{rem:awkward} can be restated as follows: The isomorphisms \eqref{eqn:awkward} do not necessarily constitute an isomorphism of persistence modules as in Definition~\ref{def:morphism}.
\end{remark}

\begin{remark}
Proposition~\ref{prop:CH} and Definition~\ref{def:zetapm} allow one to associate a zeta function to any nondegenerate Liouville domain. Note that this is not the same as the zeta function in Definition~\ref{def:zetanondegenerate}; the relation between these two zeta functions will be explained in \S\ref{sec:algebraic}.
\end{remark}



\section{Filtered equivariant symplectic homology of general Liouville domains}
\label{sec:inverselimit}

We now use an inverse limit construction to extend the definition of filtered equivariant symplectic homology $CH^L(X,\lambda)$, and the maps $CH^{L_2,L_1}(X,\lambda)$ for $L_1\le L_2$, from nondegenerate Liouville domains to arbitrary Liouville domains. (There are some other inverse limit constructions in nonequivariant symplectic homology in \cite{fh} and \cite[\S9.2]{polterovich}.) We then prove that this structure is invariant under exact symplectomorphisms of the interior of the Liouville domain.

\subsection{The inverse limit construction}
\label{subsec:ilc}

\begin{definition}
\label{def:emb}
If $(X,\lambda)$ is a Liouville domain, let $\op{Emb}(X,\lambda)$ denote the set\footnote{To avoid set-theoretic difficulties, we can assume that all manifolds under consideration are submanifolds of $\R^N$ for some $N$.} of exact symplectic embeddings
\begin{equation}
\label{eqn:varphi}
\varphi: (W,\mu) \hookrightarrow (\op{int}(X),\lambda)
\end{equation}
where $(W,\mu)$ is a nondegenerate Liouville domain.
\end{definition}

\begin{definition}
\label{def:partialorder}
If $(X,\lambda)$ is a Liouville domain, define a partial order on the set $\op{Emb}(X,\lambda)$ as follows. Given $\varphi_1:(W_1,\mu_1)\hookrightarrow (\op{int}(X),\lambda)$ and $\varphi_2:(W_2,\mu_2)\hookrightarrow (\op{int}(X),\lambda)$ in $\op{Emb}(X,\lambda)$,  we declare that $\varphi_1 < \varphi_2$ if
\[
\varphi_1(W_1) \subset \op{int}(\varphi_2(W_2)).
\]
In this situation, since $\varphi_1$ and $\varphi_2$ are embeddings, there is a unique embedding
\[
\psi_{\varphi_2,\varphi_1}: W_1 \hookrightarrow \op{int}(W_2)
\]
such that
\begin{equation}
\label{eqn:defpsi}
\varphi_1 = \varphi_2\circ\psi_{\varphi_2,\varphi_1}: W_1 \hookrightarrow \op{int}(X).
\end{equation}
One can check that $\psi_{\varphi_2,\varphi_1}$ is an exact symplectic embedding $(W_1,\mu_1)\hookrightarrow (\op{int}(W_2),\mu_2)$.
\end{definition}

\begin{lemma}
\label{lem:directed}
If $(X,\lambda)$ is a Liouville domain, then the set $\op{Emb}(X,\lambda)$ is nonempty, and any two elements of $\op{Emb}(X,\lambda)$ have a common strict upper bound. In particular, the set $\op{Emb}(X,\lambda)$, together with the partial order $<$, is a directed set.
\end{lemma}

\begin{proof}
If $\delta>0$, let $X_\delta\subset \op{int}(X)$ denote the image of the map $\Phi_X^\delta$ in Example~\ref{ex:scaling}. The sets $X_\delta$ for $\delta>0$ satisfy $X_{\delta_2}\subset X_{\delta_1}$ when $\delta_1\le \delta_2$ and exhaust $\op{int}(X)$. It follows that if $\varphi:(W,\mu)\hookrightarrow (\op{int}(X),\lambda)$ is an element of $\op{Emb}(X,\lambda)$, then we can choose $\delta>0$ sufficiently small with respect to $\varphi$ so that $\varphi(W)\subset \op{int}(X_\delta)$.

If $(X,\lambda)$ is nondegenerate, then $\Phi_X^\delta\in\op{Emb}(X,\lambda)$ and $\varphi < \Phi_X^\delta$. If $(X,\lambda)$ is degenerate, then we can find a subset $\widehat{X}\subset X$ whose boundary is a perturbation of $\partial X$ such that $(\widehat{X},\lambda)$ is nondegenerate (see Lemma~\ref{lem:Lapproxexist} below), and $\varphi < \Phi_X^\delta\big|_{(\widehat{X},e^{-\delta}\lambda)}$.
\end{proof}

On the directed set $(\op{Emb}(X,\lambda),<)$, filtered equivariant symplectic homology and transfer morphisms define an inverse system of $\Z/2$-graded vector spaces. This inverse system assigns to each element $\varphi:(W,\mu)\hookrightarrow(\op{int}(X),\lambda)$ of $\op{Emb}(X,\lambda)$ the $\Z/2$-graded vector space $CH^L(W,\mu)$. If $\varphi_1,\varphi_2\in\op{Emb}(X,\lambda)$ with $\varphi_1<\varphi_2$ as in Definition~\ref{def:partialorder}, then the inverse system assigns to the pair $(\varphi_1,\varphi_2)$ the transfer morphism
\begin{equation}
\label{eqn:inversetransfer}
CH^L(\psi_{\varphi_2,\varphi_1}) : CH^L(W_2,\mu_2) \longrightarrow CH^L(W_1,\mu_1).
\end{equation}

It now makes sense to take the inverse limit of this inverse system.

\begin{definition}
\label{def:mathring}
If $(X,\lambda)$ is a Liouville domain and $L\in\R$, define
\[
\mathring{CH}^L(X,\lambda) = \varprojlim \left\{CH^L(W,\mu) \;\big|\; (\varphi:(W,\mu)\hookrightarrow (\op{int}(X),\lambda)) \in\op{Emb}(X,\lambda) \right\}.
\]
\end{definition}

Concretely, an element of $\mathring{CH}^L(X,\lambda)$ consists of a function assigning, to each map $\varphi:(W,\mu)\hookrightarrow (\op{int}(X),\lambda)$ in $\op{Emb}(X,\lambda)$, an element $x_\varphi\in CH^L(W,\mu)$, such that if $\varphi_1 < \varphi_2$ as in Definition~\ref{def:partialorder}, then
\[
x_{\varphi_1} = CH^L(\psi_{\varphi_2,\varphi_1})(x_{\varphi_2}) \in CH^L(W_1,\mu_1).
\]

\begin{remark}
The inverse limit $\mathring{CH}^L(X,\lambda)$ is a $\Z/2$-graded vector space, but it might not be finite dimensional. However it is finite dimensional when $(X,\lambda)$ is nondegenerate, as we will see in Proposition~\ref{prop:rhoL} below, and more generally when $e^L$ is not an accumulation point of $\op{Spec}(\partial X,\lambda)$, as we will see in Proposition~\ref{prop:approx}.
\end{remark}

Now suppose that $L_1\le L_2$. Then for an embedding $\varphi:(W,\mu)\hookrightarrow (\op{int}(X),\lambda)$ in $\op{Emb}(X,\lambda)$, we have a persistence morphism
\begin{equation}
\label{eqn:openpersistence}
CH^{L_2,L_1}(W,\mu): CH^{L_1}(W,\mu) \longrightarrow CH^{L_2}(W,\mu).
\end{equation}
If $\varphi_1,\varphi_2\in\op{Emb}(X,\lambda)$ with $\varphi_1<\varphi_2$ as in Definition~\ref{def:partialorder}, then since transfer morphisms commute with persistence morphisms by \eqref{eqn:transferfiltration}, we have a commutative diagram
\[
\begin{CD}
CH^{L_1}(W_2,\mu_2) @>{CH^{L_1}(\psi_{\varphi_2,\varphi_1})}>> CH^{L_1}(W_1,\mu_1)\\
@V{CH^{L_2,L_1}(W_2,\mu_2)}VV  @VV{CH^{L_2,L_1}(W_1,\mu_1)}V\\
CH^{L_2}(W_2,\mu_2) @>{CH^{L_2}(\psi_{\varphi_2,\varphi_1})}>> CH^{L_2}(W_1,\mu_1).
\end{CD}
\]
This means that the maps \eqref{eqn:openpersistence} define a morphism of inverse systems.

\begin{definition}
\label{def:mathringpersistence}
If $(X,\lambda)$ is a Liouville domain and $L_1\le L_2$, define
\[
\mathring{CH}^{L_2,L_1}(X,\lambda) : \mathring{CH}^{L_1}(X,\lambda) \longrightarrow \mathring{CH}^{L_2}(X,\lambda)
\]
to be the map on inverse limits induced by the morphism of inverse systems defined by the maps \eqref{eqn:openpersistence}.
\end{definition}

It follows from Definition~\ref{def:mathringpersistence} and the Persistence property in Proposition~\ref{prop:CHL} that
\[
\mathring{CH}^{L,L}(X,\lambda) = \op{id}_{\mathring{CH}^L(X,\lambda)},
\]
and if $L_1\le L_2\le L_3$ then
\[
\mathring{CH}^{L_3,L_1}(X,\lambda) = \mathring{CH}^{L_3,L_2}(X,\lambda) \circ \mathring{CH}^{L_2,L_1}(X,\lambda).
\]

\subsection{Comparison with the nondegenerate case}

We now show that if $(X,\lambda)$ is nondegenerate, then the inverse limit construction in \S\ref{subsec:ilc} simply recovers $CH(X,\lambda)$. (However the inverse limit construction is still useful here because it enables a clean proof of invariance under exact symplectomorphisms of the interior, as we will see in \S\ref{sec:openinvariance} below.)

\begin{proposition}
\label{prop:rhoL}
Let $(X,\lambda)$ be a nondegenerate Liouville domain and let $L\in\R$. Then there is a canonical isomorphism
\[
\rho^L: \mathring{CH}^L(X,\lambda) \stackrel{\simeq}{\longrightarrow} CH^L(X,\lambda).
\]
\end{proposition}

\begin{proof}
Recall that if $\delta>0$, then we have an exact symplectic embedding
\[
\Phi_X^\delta: (X,e^{-\delta}\lambda)\hookrightarrow (X,\lambda)
\]
which is in $\op{Emb}(X,\lambda)$. Moreover, if $\delta_1 < \delta_2$, then $\Phi_X^{\delta_2} < \Phi_X^{\delta_1}$. Thus we have an inclusion of the set of $\delta>0$ (with the reverse of the usual ordering) into the directed set $\op{Emb}(X,\lambda)$. There is then a restriction map of inverse limits
\begin{equation}
\label{eqn:restriction}
\mathring{CH}^L(X,\lambda) \longrightarrow \varprojlim \left\{CH^L(X,e^{-\delta}\lambda) \;\big|\; \delta>0\right\}.
\end{equation}

As noted in the proof of Lemma~\ref{lem:directed}, if $\varphi\in\op{Emb}(X,\lambda)$, then $\varphi < \Phi_X^\delta$ whenever $\delta$ is sufficiently small with respect to $\varphi$, i.e.\ the set $\{\Phi_X^\delta\mid\delta>0\}$ is cofinal in $\op{Emb}(X,\lambda)$. It follows that the restriction map \eqref{eqn:restriction} is an isomorphism.

To complete the proof of the proposition, we now show that there is a canonical isomorphism
\begin{equation}
\label{eqn:doit}
CH^L(X,\lambda)
\stackrel{\simeq}{\longrightarrow}
\varprojlim\left\{CH^L(X,e^{-\delta}\lambda) \;\big|\; \delta>0\right\}.
\end{equation}
We define the map \eqref{eqn:doit} to be the map induced by the transfer morphisms
\begin{equation}
\label{eqn:doit2}
CH^L(\Phi_X^\delta) : CH^L(X,\lambda) \longrightarrow CH^L(X,e^{-\delta}\lambda)
\end{equation}
for each $\delta>0$. To prove that the resulting map \eqref{eqn:doit} is an isomorphism, it is enough to show that the map \eqref{eqn:doit2} is an isomorphism for $\delta>0$ sufficiently small. Since $\op{Spec}(\partial X,\lambda)$ is discrete, if $\delta>0$ is sufficiently small, then
\[
\left(e^L,e^{L+\delta}\right]\cap \op{Spec}(\partial X,\lambda) = \emptyset.
\]
For such $\delta$, it follows from equation \eqref{eqn:scaling3} and Lemma~\ref{lem:intervaliso} that the map \eqref{eqn:doit2} is an isomorphism.
\end{proof}

\begin{proposition}
\label{prop:rhoLpersistence}
Let $(X,\lambda)$ be a nondegenerate Liouville domain and let $L_1\le L_2$. Then the diagram
\[
\begin{CD}
\mathring{CH}^{L_1}(X,\lambda) @>{\rho^{L_1}}>{\simeq}> CH^{L_1}(X,\lambda)\\
@V{\mathring{CH}^{L_2,L_1}(X,\lambda)}VV @VV{CH^{L_2,L_1}(X,\lambda)}V\\
\mathring{CH}^{L_2}(X,\lambda) @>{\rho^{L_2}}>{\simeq}> CH^{L_2}(X,\lambda)
\end{CD}
\]
commutes.
\end{proposition}

\begin{proof}
Given $\delta>0$, consider the diagram
\[
\begin{tikzcd}[column sep=large]
\mathring{CH}^{L_1}(X,\lambda)
\arrow[r] \arrow[d, "{\mathring{CH}^{L_2,L_1}(X,\lambda)}"'] \arrow[rr, bend left=20, "\rho^{L_1}"'] & 
CH^{L_1}(X,e^{-\delta}\lambda)
\arrow[d, "{CH^{L_2,L_1}(X,e^{-\delta}\lambda)}"'] & 
CH^{L_1}(X,\lambda) 
\arrow[l, "CH^{L_1}(\Phi_X^\delta)"] \arrow[d, "{CH^{L_2,L_1}(X,\lambda)}"]
\\
\mathring{CH}^{L_2}(X,\lambda)
\arrow[r] \arrow[rr, bend right=20, "\rho^{L_2}"] &
CH^{L_2}(X,e^{-\delta}\lambda) &
CH^{L_2}(X,\lambda)
\arrow[l, "CH^{L_2}(\Phi_X^\delta)"']
\end{tikzcd}
\]
Here the horizontal arrow on the upper left inputs an element of the inverse limit $\mathring{CH}^{L_1}(X,\lambda)$ and evaluates it on the embedding $\Phi_X^\delta$ in $\op{Emb}(X,\lambda)$. The horizontal arrow on the lower left is defined the same way but with $L_2$ in place of $L_1$. The upper and lower triangles commute by the definition of $\rho^L$. The left square commutes by Definition~\ref{def:mathringpersistence}. The right square commutes because transfer morphisms commute with persistence maps by \eqref{eqn:transferfiltration}. To complete the proof of the proposition, we can take $\delta>0$ sufficiently small as in the proof of Proposition~\ref{prop:rhoL} so that the horizontal arrows on the right hand side are isomorphisms.
\end{proof}

\begin{remark}
Propositions~\ref{prop:rhoL} and \ref{prop:rhoLpersistence} show that if $(X,\lambda)$ is a nondegenerate Liouville domain, then $\mathring{CH}^L(X,\lambda)$, as defined in Definition~\ref{def:mathring}, and $\mathring{CH}^{L_2,L_1}(X,\lambda)$ as defined in Definition~\ref{def:mathringpersistence}, constitute a persistence module which is canonically isomorphic to $CH(X,\lambda)$. In the degenerate case, $\mathring{CH}^L(X,\lambda)$ and $\mathring{CH}^{L_2,L_1}(X,\lambda)$ might not satisfy the finite dimensionality or properness requirements in Definition~\ref{def:pm} for a persistence module.
\end{remark}


\subsection{Invariance under exact symplectomorphisms of the interior}
\label{sec:openinvariance}

The following proposition defines ``transfer morphisms'' for $\mathring{CH}^L$:

\begin{proposition}
\label{prop:openmorphism}
Let $(X,\lambda)$ and $(X',\lambda')$ be Liouville domains. Then an exact symplectic embedding
\[
\phi: (\op{int}(X),d\lambda) \hookrightarrow (\op{int}(X'),d\lambda')
\]
induces maps
\begin{equation}
\label{eqn:openmorphism}
\mathring{CH}^L(\phi) : \mathring{CH}^L(X',\lambda') \longrightarrow \mathring{CH}^L(X,\lambda)
\end{equation}
for each $L\in\R$ with the following properties:
\begin{description}
\item{(a)}
If $L_1\le L_2$, then the diagram
\begin{equation}
\label{eqn:mathringpercom}
\begin{CD}
\mathring{CH}^{L_1}(X',\lambda') @>{\mathring{CH}^{L_1}(\phi)}>> \mathring{CH}^{L_1}(X,\lambda)\\
@V{\mathring{CH}^{L_2,L_1}(X',\lambda')}VV @VV{\mathring{CH}^{L_2,L_1}(X,\lambda)}V\\
\mathring{CH}^{L_2}(X',\lambda') @>{\mathring{CH}^{L_2}(\phi)}>> \mathring{CH}^{L_2}(X,\lambda)
\end{CD}
\end{equation}
commutes.
\item{(b)}
If $X=X'$ and $\phi=\op{id}_X$, then $\mathring{CH}^L(\op{id}_X) = \op{id}_{\mathring{CH}^L(X,\lambda)}$.
\item{(c)}
If $(X'',\lambda'')$ is another Liouville domain, and if 
\[
\phi': (\op{int}(X'),d\lambda') \hookrightarrow (\op{int}(X''),d\lambda'')
\]
is another exact symplectic embedding, then
\begin{equation}
\label{eqn:openmorphismcomposition}
\mathring{CH}^L(\phi'\circ\phi) =
\mathring{CH}^L(\phi) \circ
\mathring{CH}^L(\phi') :
\mathring{CH}^L(X'',\lambda'') \longrightarrow \mathring{CH}^L(X,\lambda).
\end{equation}
\end{description}
\end{proposition}

\begin{corollary}
\label{cor:openinvariance}
If $(X,\lambda)$ and $(X',\lambda')$ are nondegenerate Liouville domains, then an exact symplectomorphism $(\op{int}(X),d\lambda)\stackrel{\simeq}{\to}(\op{int}(X'),d\lambda')$ induces an isomorphism of persistence modules $CH(X',\lambda')\stackrel{\simeq}{\to} CH(X,\lambda)$.
\end{corollary}

\begin{proof}
By Proposition~\ref{prop:openmorphism}, $\phi$ induces an isomorphism
\begin{equation}
\label{eqn:phiii}
\mathring{CH}^L(\phi): \mathring{CH}^L(X',\lambda')\stackrel{\simeq}{\longrightarrow} \mathring{CH}^L(X,\lambda)
\end{equation}
commuting with the persistence morphisms as in \eqref{eqn:mathringpercom}. By Propositions~\ref{prop:rhoL} and \ref{prop:rhoLpersistence}, this gives an isomorphism of persistence modules $CH(X',\lambda')\stackrel{\simeq}{\to} CH(X,\lambda)$.
\end{proof}

\begin{remark}
\label{rem:mathringnotpersistent}
In Corollary~\ref{cor:openinvariance}, without the nondegeneracy hypothesis, we still get isomorphisms \eqref{eqn:phiii} for each $L$, which commute with the maps in Definition~\ref{def:mathringpersistence}.
\end{remark}

\begin{proof}[Proof of Proposition~\ref{prop:openmorphism}.] We proceed in four steps.

{\em Step 1.\/} We first define the map \eqref{eqn:openmorphism}.

Let $x'\in\mathring{CH}^L(X',\lambda')$. The element $x'$ assigns to each embedding $\varphi':(W',\mu')\hookrightarrow (\op{int}(X'),\lambda')$ in $\op{Emb}(X',\lambda')$ an element $x'_{\varphi'}\in CH^L(W',\mu')$, such that if $\varphi'_1,\varphi'_2\in\op{Emb}(X',\lambda')$ with $\varphi'_1 < \varphi'_2$, then
\begin{equation}
\label{eqn:x'consistent}
x'_{\varphi'_1} = CH^L(\psi_{\varphi'_2,\varphi'_1})(x'_{\varphi'_2}).
\end{equation}
We want to define an element $x=\mathring{CH}^L(\phi)(x')$, which will associate to each embedding $\varphi:(W,\mu)\hookrightarrow (\op{int}(X),\lambda)$ in $\op{Emb}(X,\lambda)$ an element $x_\varphi\in CH^L(W,\mu)$, such that if $\varphi_1,\varphi_2\in\op{Emb}(X,\lambda)$ with $\varphi_1 < \varphi_2$, then
\begin{equation}
\label{eqn:xconsistent}
x_{\varphi_1} = CH^L(\psi_{\varphi_2,\varphi_1})(x_{\varphi_2}).
\end{equation}

Given an embedding $\varphi:(W,\mu)\hookrightarrow (\op{int}(X),\lambda)$ in $\op{Emb}(X,\lambda)$, we define $x_\varphi\in CH^L(W,\mu)$ as follows. Observe that $\phi\circ\varphi\in\op{Emb}(X',\lambda')$. By Lemma~\ref{lem:directed}, we can find an embedding $\varphi':(W',\mu')\hookrightarrow (\op{int}(X'),\lambda')$ in $\op{Emb}(X',\lambda')$ with $\phi\circ\varphi < \varphi'$. This means that we have an exact symplectic embedding
\[
\psi_{\varphi',\phi\circ\varphi} : (W,\mu) \hookrightarrow (\op{int}(W'),\mu').
\]
We now define
\begin{equation}
\label{eqn:defxvarphi}
x_\varphi = CH^L(\psi_{\varphi',\phi\circ\varphi})(x_{\varphi'}') \in CH^L(W,\mu).
\end{equation}
To show that this gives a well defined element of $\mathring{CH}^L(X,\lambda)$, we need to check that (i) $x_\varphi$ does not depend on the choice of $\varphi'$, and (ii) the consistency condition \eqref{eqn:xconsistent} holds.

To prove (i), let $\varphi'_1:(W_1',\mu_1')\hookrightarrow (\op{int}(X'),\lambda')$ and $\varphi'_2:(W_2',\mu_2')\hookrightarrow (\op{int}(X'),\lambda')$ be two elements of $\op{Emb}(X',\lambda')$ with $\phi\circ\varphi < \varphi'_1,\varphi'_2$. By Lemma~\ref{lem:directed}, there exists an embedding $\varphi'_3:(W_3',\mu_3')\hookrightarrow (\op{int}(X'),\lambda')$ in $\op{Emb}(X',\lambda')$ with $\varphi'_1,\varphi'_2 < \varphi'_3$. We have a commutative diagram of embeddings
\[
\begin{CD}
W @>{\psi_{\varphi_1',\phi\circ\varphi}}>> W_1'\\
 @V{\psi_{\varphi_2',\phi\circ\varphi}}VV  @VV{\psi_{\varphi_3',\varphi_1'}}V\\
 W_2' @>{\psi_{\varphi_3',\varphi_2'}}>> W_3'.
\end{CD}
\]
Both compositions in the diagram agree with the exact symplectic embedding
\[
\psi_{\varphi_3',\phi\circ\varphi} : (W,\mu) \hookrightarrow (W_3',\mu_3').
\]
It follows from this and the functoriality of transfer morphisms that
\[
CH^L(\psi_{\varphi_1',\phi\circ\varphi})(x_{\varphi_1'}') = CH^L(\psi_{\varphi_2',\phi\circ\varphi})(x_{\varphi_2'}') = CH^L(\psi_{\varphi_3',\phi\circ\varphi})(x_{\varphi_3'}') \in CH^L(W,\mu).
\]

Assertion (ii) follows from a similar argument using the functoriality of transfer morphisms.

{\em Step 2.\/} We now prove that the maps $\mathring{CH}^L(\phi)$ satisfy property (a).  By Definition~\ref{def:mathringpersistence}, we need to show that for $\varphi$ and $\varphi'$ as in Step 1, the diagram
\[
\begin{CD}
CH^{L_1}(W',\mu') @>{CH^{L_1}(\psi_{\varphi',\phi\circ\varphi})}>> CH^{L_1}(W,\mu)\\
@V{CH^{L_2,L_1}(W',\mu')}VV @VV{CH^{L_2,L_1}(W,\mu)}V\\
CH^{L_2}(W',\mu') @>{CH^{L_2}(\psi_{\varphi',\phi\circ\varphi})}>> CH^{L_2}(W,\mu)
\end{CD}
\]
commutes. This holds because transfer morphisms commute with persistence maps \eqref{eqn:transferfiltration}.

{\em Step 3.\/} We now prove property (b).
Let
\[
x'=\left\{x'_{\varphi'}\;\big|\;\varphi'\in\op{Emb}(X,\lambda)\right\}\in \mathring{CH}^L(X,\lambda).
\]
Write
\[
\mathring{CH}^L(\op{id}_X)(x') = \left\{x_\varphi \;\big|\; \varphi\in\op{Emb}(X,\lambda)\right\} \in \mathring{CH}^L(X,\lambda).
\]
If $\varphi\in\op{Emb}(X,\lambda)$, then by Lemma~\ref{lem:directed} we can find $\varphi'\in\op{Emb}(X,\lambda)$ with $\varphi < \varphi'$, and by the definition \eqref{eqn:defxvarphi}, we have
\[
x_\varphi = CH^L(\psi_{\varphi',\varphi})(x'_{\varphi'}).
\]
This agrees with $x'_{\varphi}$ by the consistency property \eqref{eqn:x'consistent}.

{\em Step 4.\/} We now prove property (c). Let
\[
x''=\left\{x''_{\varphi''}\;\big|\;\varphi''\in\op{Emb}(X'',\lambda'')\right\}\in \mathring{CH}^L(X'',\lambda'').
\]
Write
\[
\begin{split}
\mathring{CH}^L(\phi')(x'') &= \left\{x'_{\varphi'}\;\big|\;\varphi'\in\op{Emb}(X',\lambda')\right\}\in \mathring{CH}^L(X',\lambda'),\\
\mathring{CH}^L(\phi)(\mathring{CH}^L(\phi')(x'')) &= \left\{x_{\varphi}\;\big|\;\varphi\in\op{Emb}(X,\lambda)\right\}\in \mathring{CH}^L(X,\lambda),\\
\mathring{CH}^L(\phi'\circ\phi)(x'') &= \left\{\widehat{x}_{\varphi}\;\big|\;\varphi\in\op{Emb}(X,\lambda)\right\}\in \mathring{CH}^L(X,\lambda).
\end{split}
\]
Let $\varphi:(W,\mu)\hookrightarrow (\op{int}(X),\lambda)$ be an embedding in $\op{Emb}(X,\lambda)$. We need to show that
\[
x_\varphi = \widehat{x}_\varphi \in CH^L(W,\mu).
\]
By Lemma~\ref{lem:directed}, we can find $\varphi':(W',\mu')\hookrightarrow (\op{int}(X'),\lambda')$ in $\op{Emb}(X',\lambda')$ with $\phi\circ\varphi < \varphi'$, and we can find $\varphi'':(W'',\mu'')\hookrightarrow (\op{int}(X''),\lambda'')$ in $\op{Emb}(X'',\lambda'')$ with $\phi'\circ\varphi' < \varphi''$. By the definition \eqref{eqn:defxvarphi}, we have
\begin{align}
\nonumber
x'_{\varphi'} &= CH^L(\psi_{\varphi'',\phi'\circ\varphi'})(x''_{\varphi''}) \in CH^L(W',\mu'),\\
\label{eqn:xcomp1}
x_\varphi &= CH^L(\psi_{\varphi',\phi\circ\varphi})(CH^L(\psi_{\varphi'',\phi'\circ\varphi'})(x''_{\varphi''})) \in CH^L(W,\mu).
\end{align}
On the other hand, the embedding $\phi'\circ\phi\circ\varphi: (W,\mu) \hookrightarrow (\op{int}(X''),\lambda'')$ is an element of $\op{Emb}(X'',\lambda'')$ with $\phi'\circ\phi\circ\varphi < \varphi''$. So by the definition \eqref{eqn:defxvarphi} again, we have
\begin{equation}
\label{eqn:xcomp2}
\widehat{x}_\varphi = CH^L(\psi_{\varphi'',\phi'\circ\phi\circ\varphi})(x''_{\varphi''}) \in CH^L(W,\mu).
\end{equation}
Finally, it follows from equation \eqref{eqn:defpsi} that
\[
\psi_{\varphi'',\phi'\circ\varphi'} \circ \psi_{\varphi',\phi\circ\varphi} = \psi_{\varphi'',\phi'\circ\phi\circ\varphi} : W \longrightarrow W''.
\]
Then by the functoriality of transfer morphisms,
\[
CH^L(\psi_{\varphi',\phi\circ\varphi}) \circ CH^L(\psi_{\varphi'',\phi'\circ\varphi'}) = CH^L(\psi_{\varphi'',\phi'\circ\phi\circ\varphi}) : CH^L(W'',\mu'') \longrightarrow CH^L(W,\mu).
\]
It follows from this that \eqref{eqn:xcomp1} and \eqref{eqn:xcomp2} agree.
\end{proof}

\begin{example}
Suppose that $(X,\lambda)$ and $(X',\lambda')$ are nondegenerate Liouville domains, $\varphi:(X,\lambda)\hookrightarrow (\op{int}(X'),\lambda')$ is an exact symplectic embedding, and $\phi=\varphi|_{\op{int}(X)}$. Then $\mathring{CH}^L(\phi)$ agrees with the usual transfer morphism. That is, the diagram
\[
\begin{CD}
\mathring{CH}^L(X',\lambda') @>{\mathring{CH}^L(\phi)}>> \mathring{CH}^L(X,\lambda)\\
@V{\rho^L}V{\simeq}V @V{\rho^L}V{\simeq}V \\
CH^L(X',\lambda') @>{CH^L(\varphi)}>> CH^L(X,\lambda)
\end{CD}
\]
commutes. This follows from the definitions together with the functoriality of transfer morphisms.
\end{example}

\begin{remark}
One can use a similar inverse limit construction to define filtered equivariant symplectic homology (and prove its invariance under exact symplectomorphisms) for open subsets of Liouville domains with nonsmooth boundary (e.g.\ subsets of $\R^{2n}$ with piecewise smooth boundary), provided that they are exhausted by nondegenerate Liouville domains. One can also repeat this construction using other versions of filtered contact homology, such as filtered embedded contact homology in the four-dimensional case \cite{bn}.
\end{remark}



\section{Invariance of the zeta function in the nondegenerate case}
\label{sec:algebraic}

We now prove Theorem~\ref{thm:nondegenerate}.

Let $(X,\lambda)$ be a nondegenerate Liouville domain. By Proposition~\ref{prop:CH}, there is an associated persistence module $CH(X,\lambda)$. As explained in \S\ref{sec:zetapm}, every persistence module has an associated zeta function. We can then make the following definition:

\begin{definition}
If $(X,\lambda)$ is a nondegenerate Liouville domain, define
\[
\zeta_{CH}(X,\lambda) = \zeta(CH(X,\lambda)) \in \Lambda.
\]
\end{definition}

If $(X',\lambda')$ is another nondegenerate Liouville domain, and if there exists an exact symplectomorphism $(\op{int}(X),d\lambda) \stackrel{\simeq}{\to} (\op{int}(X'),d\lambda')$, then by Corollary~\ref{cor:openinvariance}, we have an isomorphism of persistence modules $CH(X,\lambda) \simeq CH(X',\lambda')$. The zeta function of a persistence module is by definition invariant under isomorphism of persistence modules, so we obtain
\begin{equation}
\label{eqn:zetachinv}
\zeta_{CH}(X,\lambda) = \zeta_{CH}(X',\lambda').
\end{equation}

This does not yet prove Theorem~\ref{thm:nondegenerate}, because the zeta function $\zeta_{CH}$ defined above is not the zeta function that appears in Theorem~\ref{thm:nondegenerate}, which was defined in Definition~\ref{def:zetanondegenerate}. However we can obtain the latter from the former by an algebraic transformation. (Some related algebra appears in \cite{ionelparker}.)

Recall that the M\"obius function $\mu:\Z^{>0}\to\{-1,0,1\}$ is defined as follows: If $n$ is divisible by the square of a prime number then $\mu(n)=0$. Otherwise, $\mu(n)$ is $-1$ to the number of prime factors of $n$.

\begin{definition}
\label{def:Thetadefinition}
Define a function
\[
\Theta: \Lambda \longrightarrow 1 + \Lambda^+
\]
as follows. If $\alpha\in\Lambda$, regarded as a function $\R\to\Z$, then
\begin{equation}
\label{eqn:Thetaproduct}
\Theta(\alpha) = \prod_{s\in\R}\prod_{n=1}^\infty \left(1-t^{ne^s}\right)^{-\alpha(s)\mu(n)}.
\end{equation}
\end{definition}

\begin{lemma}
\label{lem:Thetadefined}
The function $\Theta$ is well-defined.
\end{lemma}

\begin{proof}
Whenever $\alpha(s)\mu(n)>0$, we expand the corresponding factor in the product \eqref{eqn:Thetaproduct} using equation \eqref{eqn:novikovinverse}. Since $\alpha\in\Lambda$, the function $\alpha$ is bounded from below, i.e.\ there exists $L_0>0$ such that $\alpha(s)=0$ whenever $e^s<L_0$. Then each factor in \eqref{eqn:Thetaproduct} has the form
\[
1 + c_1t^\sigma + c_2t^{2\sigma} + \cdots
\]
where $c_1,c_2,\ldots\in\Z$ with $c_1\neq 0$ and $\sigma \ge L_0$. For any $L>0$, there are only finitely many such factors with $\sigma < L$. It follows that the product \eqref{eqn:Thetaproduct} has only finitely many monomials $ct^\sigma$ with $c\neq 0$ and $\sigma < L$.
\end{proof}

\begin{remark}
\label{rem:multiplicative}
It follows from the definition that $\Theta$ is multiplicative: if $\alpha,\beta\in\Lambda$ then
\[
\Theta(\alpha+\beta)=\Theta(\alpha)\Theta(\beta).
\]
\end{remark}

\begin{lemma}
\label{lem:algebraic}
If $(X,\lambda)$ is a nondegenerate Liouville domain, then
\[
\Theta(\zeta_{CH}(X,\lambda)) = \zeta(X,\lambda).
\]
\end{lemma}

As a warmup to proving this, we first give:

\begin{proof}[Proof of Lemma~\ref{lem:productformula}]
Since every Reeb orbit is a cover of a simple Reeb orbit, we can rewrite the sum in \eqref{eqn:defzeta} as a sum over simple Reeb orbits to get
\[
\zeta(Y,\lambda) = \exp \sum_{\gamma\in\mathcal{P}_{\op{simp}}(Y,\lambda)}\sum_{d=1}^\infty \frac{(-1)^{\epsilon(\gamma^d)}}{d}t^{d\mathcal{A}(\gamma)}.
\]
Then to prove the lemma, it is enough to show that if $\gamma\in\mathcal{P}_{\op{simp}}(Y,\lambda)$ is a simple Reeb orbit, then
\[
\exp \sum_{d=1}^\infty \frac{(-1)^{\epsilon(\gamma^d)}}{d}t^{d\mathcal{A}(\gamma)} = \left(1-(-1)^{\epsilon(\gamma)+\epsilon(\gamma^2)} t^{\mathcal{A}(\gamma)}\right)^{-(-1)^{\epsilon(\gamma^2)}}.
\]
Taking the formal logarithm of both sides, this is equivalent to
\[
\sum_{d=1}^\infty \frac{(-1)^{\epsilon(\gamma^d)}}{d}t^{d\mathcal{A}(\gamma)} = (-1)^{\epsilon(\gamma^2)}\sum_{d=1}^\infty \frac{\left((-1)^{\epsilon(\gamma)+\epsilon(\gamma^2)} t^{\mathcal{A}(\gamma)}\right)^d}{d}.
\]
Thus it is enough to show that for each positive integer $d$ we have
\[
(-1)^{\epsilon(\gamma^d)} = (-1)^{d\epsilon(\gamma)+(d+1)\epsilon(\gamma^2)}.
\]
Equivalently,
\begin{equation}
\label{eqn:iterateparity}
\epsilon(\gamma^d) = \left\{\begin{array}{cl} \epsilon(\gamma), & \mbox{$d$ odd},\\
\epsilon(\gamma^2), & \mbox{$d$ even}.
\end{array}\right.
\end{equation}
This holds because if $d$ is odd then $\epsilon(\gamma^d)$ is the parity of the number of eigenvalues of the linearized return map $P_\gamma$ in the interval $(0,1)$, while if $d$ is even then $\epsilon(\gamma^d)$ is the parity of the number of eigenvalues of $P_\gamma$ in the interval $(-1,1)$.
\end{proof}

\begin{proof}[Proof of Lemma~\ref{lem:algebraic}.]
By equations \eqref{eqn:zetapm}, \eqref{eqn:Pgood}, and \eqref{eqn:ecj}, we have
\[
\zeta_{CH}(X,\lambda) = \sum_{\gamma\in\mathcal{P}_{\op{good}}(\partial X,\lambda)} (-1)^{\epsilon(\gamma)}t^{\log(\mathcal{A}(\gamma))}.
\]
We can write this as a sum over simple Reeb orbits as follows:
\[
\zeta_{CH}(X,\lambda) = \sum_{\gamma\in\mathcal{P}_{\op{simp}}(\partial X,\lambda)} \sum_{\substack{\mbox{\scriptsize $d\ge 1$}\\ \mbox{\scriptsize $\gamma^d$ good}}}
(-1)^{\epsilon(\gamma^d)}t^{\log(d\mathcal{A}(\gamma))}.
\]
It follows from equation \eqref{eqn:Thetaproduct} and Remark~\ref{rem:multiplicative} that
\[
\begin{split}
\Theta(\zeta_{CH}(X,\lambda))
&=
\prod_{\gamma\in\mathcal{P}_{\op{simp}}(\partial X,\lambda)}
\prod_{\substack{\mbox{\scriptsize $d\ge 1$}\\ \mbox{\scriptsize $\gamma^d$ good}}}
\prod_{n=1}^\infty
\left(
1-\left( t^{\mathcal{A}(\gamma)} \right)^{nd}
\right)^{-\mu(n)(-1)^{\epsilon(\gamma^d)}}.
\end{split}
\]
By the product formula \eqref{eqn:productformula} for the zeta function, it is enough to show that if $\gamma$ is a simple Reeb orbit, then writing $z=t^{\mathcal{A}(\gamma)}$, we have
\begin{equation}
\label{eqn:algebraic}
\left(1-(-1)^{\epsilon(\gamma)+\epsilon(\gamma^2)} z\right)^{-(-1)^{\epsilon(\gamma^2)}}
=
\prod_{\substack{\mbox{\scriptsize $d\ge 1$}\\ \mbox{\scriptsize $\gamma^d$ good}}}
\prod_{n=1}^\infty
\left(
1 - z^{nd}
\right)^{-\mu(n)(-1)^{\epsilon(\gamma^d)}}.
\end{equation}

To prove equation \eqref{eqn:algebraic}, there are four cases, depending on the parities $\epsilon(\gamma),\epsilon(\gamma^2)\in\Z/2$.

{\em Case 1:\/} $\epsilon(\gamma)=\epsilon(\gamma^2)=0$.

In this case we have $\epsilon(\gamma^d)=0$ for all $d\ge 1$ by equation \eqref{eqn:iterateparity}. In particular, $\gamma^d$ is good for all $d$ by Definition~\ref{def:good}. Thus equation \eqref{eqn:algebraic} becomes
\begin{equation}
\label{eqn:algebraic1}
(1-z)^{-1} = \prod_{d\ge 1}\prod_{n=1}^\infty (1-z^{nd})^{-\mu(n)}.
\end{equation}
To prove this, we can rewrite the right hand side as
\[
\prod_{k\ge 1}(1-z^k)^{-\sum_{n|k}\mu(n)}.
\]
Equation \eqref{eqn:algebraic1} then reduces to the identity
\[
\sum_{n|k}\mu(n) = \left\{\begin{array}{cl} 1, & k=1,\\ 0, & k>1, \end{array}\right.
\]
which follows from the M\"obius inversion formula or can be checked directly.

{\em Case 2:\/} $\epsilon(\gamma)=\epsilon(\gamma^2)=1$.

In this case, equation \eqref{eqn:algebraic} reduces to the same identity as \eqref{eqn:algebraic1}, except that both sides are inverted.

{\em Case 3:\/} $\epsilon(\gamma)=0$ and $\epsilon(\gamma^2)=1$.

In this case, by equation \eqref{eqn:iterateparity}, $\epsilon(\gamma^d)$ alternates, so $\gamma^d$ is good if and only if $d$ is odd. Thus equation \eqref{eqn:algebraic} becomes
\begin{equation}
\label{eqn:algebraic3}
1 + z = \prod_{\mbox{\scriptsize $d$ odd}}
\prod_{n=1}^\infty
(1-z^{nd})^{-\mu(n)}.
\end{equation}
Since
\[
1 + z = (1-z)^{-1}(1-z^2),
\]
equation \eqref{eqn:algebraic3} follows by starting with equation \eqref{eqn:algebraic1}, and dividing by equation \eqref{eqn:algebraic1} with $z$ replaced by $z^2$.

{\em Case 4:\/} $\epsilon(\gamma)=1$ and $\epsilon(\gamma^2)=0$.

In this case, equation \eqref{eqn:algebraic} reduces to the same identity as \eqref{eqn:algebraic3}, except that both sides are inverted.
\end{proof}

By equation \eqref{eqn:zetachinv} and Lemma~\ref{lem:algebraic}, the proof of Theorem~\ref{thm:nondegenerate} is complete.

\begin{remark}
In Lemma~\ref{lem:algebraic}, we do not lose any information in passing from $\zeta_{CH}(X,\lambda)$ to $\zeta(X,\lambda)$, because the function $\Theta$ is injective. To prove this injectivity, by the multiplicativity in Remark~\ref{rem:multiplicative} it is enough to show that if $\alpha\in\Lambda\setminus\{0\}$ then $\Theta(\alpha)\neq 1$. This holds because if $\alpha$ has leading term $ct^\sigma$, then $\Theta(\alpha)-1$ has leading term $ct^{e^\sigma}$.
\end{remark}



\section{The degenerate case}
\label{sec:degeneratecase}

We now develop the properties of $\mathring{CH}^L(X,\lambda)$ and $\mathring{CH}^{L_2,L_1}(X,\lambda)$ for an arbitrary Liouville domain $(X,\lambda)$, which might be degenerate. We will then define the ``dynamically tame'' condition and prove Theorem~\ref{thm:tame}.

\subsection{Approximation}

To start, we now discuss how to compute $\mathring{CH}^L(X,\lambda)$ for almost every $L$ by approximating $(X,\lambda)$ by a nondegenerate Liouville domain.

Let $V$ be the Liouville vector field on $X$, see Example~\ref{ex:scaling}. The flow of $V$ for negative time defines a smooth embedding
\[
\phi_V:(-\infty,0]\times \partial X \longrightarrow X
\]
sending $(0,y)\mapsto y$ when $y\in\partial X$. Note that if $\delta > 0$, then in the notation of \S\ref{sec:openinvariance}, we have
\[
X_\delta = X \setminus \phi_V((-\delta,0]\times\partial X).
\]

\begin{definition}
\label{def:Lapprox}
Let $(X,\lambda)$ be a Liouville domain and let $L\in\R$. Suppose that
\begin{equation}
\label{eqn:sigma}
\exists \sigma > 0 : 
\left(e^L,e^{L+\sigma}\right]\cap \op{Spec}(\partial X,\lambda) = \emptyset.
\end{equation}
Let
\begin{equation}
\label{eqn:delta_L}
\delta_L = \sup\left\{\sigma>0\;\bigg|\; 
\left(e^L,e^{L+2\sigma}\right] \cap \op{Spec}(\partial X,\lambda) = \emptyset\right\} \in (0,\infty].
\end{equation}
An {\bf $L$-approximation\/} to $(X,\lambda)$ is a closed subset $\widehat{X}\subset \op{int}(X)$ such that:
\begin{itemize}
\item $X_{\delta_L}\subset \op{int}(\widehat{X})$.
\item $\partial\widehat{X}$ is $\phi_V$ of the graph of a smooth function $\partial X \to (-\delta_L,0)$. In particular, $(\widehat{X},\lambda)$ is a Liouville domain.
\item The Liouville domain $(\widehat{X},\lambda)$ is nondegenerate.
\item $\left(e^L,e^{L+\delta_L}\right]\cap\op{Spec}(\partial \widehat{X},\lambda) = \emptyset$.
\end{itemize}
\end{definition}

\begin{remark}
Condition \eqref{eqn:sigma} holds for example if $e^L\notin\op{Spec}(\partial X,\lambda)$, since $\op{Spec}(\partial X,\lambda)$ is a closed subset of $\R$. In particular, since $\op{Spec}(\partial X,\lambda)$ has measure zero, it follows that condition \eqref{eqn:sigma} holds for almost every $L$.
\end{remark}

\begin{lemma}
\label{lem:Lapproxexist}
Let $(X,\lambda)$ be a Liouville domain, let $L\in\R$, and assume that condition \eqref{eqn:sigma} holds. Then $L$-approximations to $(X,\lambda)$ exist and exhaust $\op{int}(X)$.
\end{lemma}

\begin{proof}
Choose $\delta\in(0,\delta_L)$. Let $f:\partial X\to(-\delta_L,0)$ be a smooth function which is a perturbation of the constant function $-\delta$, and let
\[
\widehat{X} = X \setminus \{\phi_V(s,y)\mid f(y)<s\le 0\}.
\]
Then $\partial\widehat{X}$ is naturally identified with $\partial X$ so that the contact form $\lambda|_{\partial\widehat{X}}$ is identified with $e^{f}\lambda|_{\partial X}$. By a standard transversality argument, this contact form is nondegenerate for generic $f$. Moreover, if $f$ is a sufficiently small perturbation of the constant function $-\delta$, then this contact form has no Reeb orbits with action in the interval $(e^L,e^{L+\delta_L}]$. Otherwise, a compactness argument would show that $e^{-\delta}\lambda|_{\partial X}$ has a Reeb orbit with action in the interval $[e^L,e^{L+\delta_L}]$, so $\lambda|_{\partial X}$ has a Reeb orbit with action in the interval $[e^{L+\delta},e^{L+\delta+\delta_L}]$, contradicting \eqref{eqn:delta_L}.

This proves that $L$-approximations exist. Since $\delta$ in the previous paragraph can be chosen arbitrarily small, $L$-approximations exhaust $\op{int}(X)$.
\end{proof}

If $\widehat{X}$ is an $L$-approximation to $(X,\lambda)$, then from the definition of the inverse limit, there is a projection map
\begin{equation}
\label{eqn:projectionmap}
\pi_{\widehat{X}}^L : \mathring{CH}^L(X,\lambda) \longrightarrow CH^L(\widehat{X},\lambda).
\end{equation}

\begin{proposition}
\label{prop:approx}
Let $(X,\lambda)$ be a Liouville domain, let $L\in\R$, and suppose that condition \eqref{eqn:sigma} holds. Let $\widehat{X}$ be an $L$-approximation to $(X,\lambda)$. Then:
\begin{description}
\item{(a)}
The projection map \eqref{eqn:projectionmap} is an isomorphism. In particular, $\mathring{CH}^L(X,\lambda)$ is finite dimensional.
\item{(b)}
Suppose that $L_1\le L_2$ and that condition \eqref{eqn:sigma} holds for $L=L_1$ and $L=L_2$. Suppose that $\widehat{X}$ is both an $L_1$-approximation and an $L_2$-approximation to $(X,\lambda)$. Then the diagram
\[
\begin{CD}
\mathring{CH}^{L_1}(X,\lambda) @>{\pi_{\widehat{X}}^{L_1}}>{\simeq}> CH^{L_1}(\widehat{X},\lambda)\\
@V{\mathring{CH}^{L_2,L_1}(X,\lambda)}VV @VV{CH^{L_2,L_1}(\widehat{X},\lambda)}V\\
\mathring{CH}^{L_2}(X,\lambda) @>{\pi_{\widehat{X}}^{L_2}}>{\simeq}> CH^{L_2}(\widehat{X},\lambda)
\end{CD}
\]
commutes.
\end{description}
\end{proposition}

The proof of Proposition~\ref{prop:approx} will use the following lemma.

\begin{lemma}
\label{lem:approximations}
Let $(X,\lambda)$ be a Liouville domain, let $L\in\R$, and suppose that condition \eqref{eqn:sigma} holds. Let $\widehat{X}$ and $\widehat{X}'$ be $L$-approximations of $(X,\lambda)$ such that $\widehat{X}\subset\op{int}(\widehat{X}')$. Then the transfer morphism
\begin{equation}
\label{eqn:iitm}
CH^L(\widehat{X}',\lambda) \longrightarrow CH^L(\widehat{X},\lambda)
\end{equation}
induced by inclusion is an isomorphism.
\end{lemma}

\begin{proof}
Let $\delta$ denote the number $\delta_L$ in \eqref{eqn:delta_L}. Then we have inclusions
\begin{equation}
\label{eqn:inclusions}
\widehat{X}_{\delta} \subset \widehat{X}'_{\delta} \subset \widehat{X} \subset \widehat{X}'
\end{equation}
with each set including into the interior of the next. Let $\imath:\widehat{X}\hookrightarrow X$ and $\imath':\widehat{X}'\hookrightarrow X$ denote the inclusion maps.  Then it follows from \eqref{eqn:inclusions} that in $\op{Emb}(X,\lambda)$ we have
\[
\imath\circ\Phi_{\widehat{X}}^{\delta} < \imath'\circ\Phi_{\widehat{X}'}^{\delta} < \imath < \imath'.
\]
From \eqref{eqn:inversetransfer} we have induced transfer morphisms
\[
CH^L(\widehat{X}',\lambda) \longrightarrow CH^L(\widehat{X},\lambda) \longrightarrow CH^L(\widehat{X}',e^{-\delta}\lambda) \longrightarrow CH^L(\widehat{X},e^{-\delta}\lambda).
\]
The first of these morphisms is the inclusion-induced transfer morphism \eqref{eqn:iitm}.
The composition of the first two morphisms is $CH^L(\Phi_{\widehat{X}'}^{\delta})$, by the functoriality of transfer morphisms \eqref{eqn:transferfunctor}. This map is an isomorphism by equation \eqref{eqn:scaling3}, Lemma~\ref{lem:intervaliso}, and the fourth bullet point in Definition~\ref{def:Lapprox}. Likewise, the composition of the last two morphisms is $CH^L(\Phi_{\widehat{X}}^{\delta})$, which is an isomorphism. It follows that \eqref{eqn:iitm} is an isomorphism.
\end{proof}

\begin{proof}[Proof of Proposition~\ref{prop:approx}.]
(a)
Let $\op{Approx}^L(X,\lambda)$ denote the set of $L$-approximations to $(X,\lambda)$. We define a partial order on $\op{Approx}^L(X,\lambda)$ by declaring that $\widehat{X}_1 < \widehat{X}_2$ iff $\widehat{X}_1 \subset \op{int}(\widehat{X}_2)$. Since $L$-approximations exhaust $\op{int}(X)$ by Lemma~\ref{lem:Lapproxexist}, it follows that $(\op{Approx}^L(X,\lambda),<)$ is a directed set.

We define an inverse system over this directed set, by assigning to each $L$-approximation $\widehat{X}$ the $\Z/2$-graded vector space $CH^L(\widehat{X},\lambda)$, and to each pair of $L$-approximations $\widehat{X}_1,\widehat{X}_2$ with $\widehat{X}_1<\widehat{X}_2$ the transfer morphism $CH^L(\widehat{X}_2,\lambda) \to CH^L(\widehat{X}_1,\lambda)$ induced by inclusion. 

There is an inclusion of directed systems $\op{Approx}^L(X,\lambda) \to \op{Emb}(X,\lambda)$, sending an $L$-approximation $\widehat{X}$ to the inclusion map $\widehat{X}\to X$. We then have a restriction map of inverse limits
\begin{equation}
\label{eqn:rmil}
\mathring{CH}^L(X,\lambda) \longrightarrow \varprojlim \left\{ CH^L(\widehat{X},\lambda) \;\bigg|\; \widehat{X}\in \op{Approx}^L(X,\lambda)\right\}.
\end{equation}
Since $L$-approximations are cofinal in $\op{Emb}(X,\lambda)$, this restriction map is an isomorphism. For any particular $L$-approximation $\widehat{X}$, the projection map from the inverse limit on the right hand side of \eqref{eqn:rmil} to $CH^L(\widehat{X},\lambda)$ is an isomorphism by Lemma~\ref{lem:approximations}.

(b) This follows from Definition~\ref{def:mathringpersistence}.
\end{proof}


\subsection{Stability}

\begin{proposition}
\label{prop:stability}
Let $(X,\lambda)$ be a Liouville domain, and let $L_1, L_2 \in\R$ with $L_1 < L_2$. Suppose that
\begin{equation}
\label{eqn:stability}
\left(e^{L_1},e^{L_2}\right] \cap \op{Spec}(\partial X,\lambda) = \emptyset.
\end{equation}
Then the map
\[
\mathring{CH}^{L_2,L_1}(X,\lambda) : \mathring{CH}^{L_1}(X,\lambda) \longrightarrow \mathring{CH}^{L_2}(X,\lambda)
\]
is an isomorphism.
\end{proposition}

\begin{proof}
It follows from \eqref{eqn:stability} that condition \eqref{eqn:sigma} holds for $L=L_1$. Since $\op{Spec}(\partial X,\lambda)$ is a closed set, condition \eqref{eqn:sigma} also holds for $L=L_2$. By following the proof of Lemma~\ref{lem:Lapproxexist} using $\delta < \delta_{L_2}$, we can find $\widehat{X}\subset\op{int}(X)$ which is both an $L_1$-approximation and an $L_2$-approximation to $(X,\lambda)$. By Proposition~\ref{prop:approx}, it is enough to show that the persistence morphism
\[
CH^{L_2,L_1}(\widehat{X},\lambda) :CH^{L_1}(\widehat{X},\lambda) \longrightarrow CH^{L_2}(\widehat{X},\lambda)
\]
is an isomorphism. This holds by Lemma~\ref{lem:intervaliso}.
\end{proof}

\begin{remark}
It follows from Propositions~\ref{prop:approx} and \ref{prop:stability} that if $\op{Spec}(\partial X,\lambda)$ is discrete, then the $\Z/2$-graded vector spaces $\mathring{CH}^L(X,\lambda)$, and the maps $\mathring{CH}^{L_2,L_1}(X,\lambda)$ for $L_1\le L_2$, constitute a persistence module as in Definition~\ref{def:pm}.
\end{remark}


\subsection{The dynamically tame case}
\label{sec:dynamicallytame}

Recall from Proposition~\ref{prop:approx} that if $(X,\lambda)$ is a Liouville domain and if $e^L\in\R\setminus\op{Spec}(\partial X,\lambda)$, then $\mathring{CH}^L(X,\lambda)$ is a finite dimensional $\Z/2$-graded vector space.

\begin{definition}
\label{def:ecdt}
Let $(X,\lambda)$ be a Liouville domain, let $L\in\R$, and suppose that $e^L\notin\op{Spec}(\partial X,\lambda)$. Define
\[
\chi^L\left(X,\lambda\right) = \dim\left(\mathring{CH}^L_0(X,\lambda)\right) - \dim\left(\mathring{CH}^L_1(X,\lambda)\right) \in \Z.
\]
\end{definition}

\begin{definition}
\label{def:dynamicallytame}
Let $(X,\lambda)$ be a Liouville domain. We say that $(X,\lambda)$ is {\bf dynamically tame\/} if there is a closed and discrete set $S\subset\R$, which is bounded from below, such that:
\begin{description}
\item{(*)} If $L_1,L_2\in\R$ with $L_1 < L_2$ and $e^{L_1}, e^{L_2}\notin\op{Spec}(\partial X,\lambda)$, and if $[L_1, L_2]\cap S = \emptyset$, then
\[
\chi^{L_1}(X,\lambda) = \chi^{L_2}(X,\lambda).
\]
\end{description}
\end{definition}

\begin{definition}
\label{def:tameecj}
Suppose that $(X,\lambda)$ is a dynamically tame Liouville domain, and let $L\in\R$. Define
\[
\chi^{L,+}(X,\lambda) = \chi^{L+\delta}(X,\lambda) \in \Z
\]
where $\delta>0$ is small and
\begin{equation}
\label{eqn:notinspec}
e^{L+\delta}\notin\op{Spec}(\partial X,\lambda).
\end{equation}
Note that we can find arbitrarily small $\delta>0$ satisfying \eqref{eqn:notinspec}, since $\op{Spec}(\partial X,\lambda)\subset\R$ has measure zero. For such $\delta$, the Euler characteristic $\chi^{L+\delta}(X,\lambda)\in\Z$ is independent of $\delta$ when $\delta$ is sufficiently small, since $S$ is assumed to be closed and discrete in Definition~\ref{def:dynamicallytame}.

Likewise, define
\[
\chi^{L,-}(X,\lambda)=\chi^{L-\delta}(X,\lambda)
\]
where $\delta>0$ is small and $e^{L-\delta}\notin\op{Spec}(\partial X,\lambda)$.

Define the {\bf Euler characteristic jump\/}
\[
\Delta_L(\mathring{CH}(X,\lambda)) = \chi^{L,+}(X,\lambda) - \chi^{L,-}(X,\lambda) \in \Z.
\]
Note that if $S\subset\R$ is as in Definition~\ref{def:dynamicallytame}, then $\Delta_L(\mathring{CH}(X,\lambda)) = 0$ when $L\notin S$.
\end{definition}

\begin{example}
\label{ex:discretespectrum}
Suppose that $\op{Spec}(\partial X,\lambda)\subset\R$ is discrete. Then $(X,\lambda)$ is dynamically tame: the condition in Definition~\ref{def:dynamicallytame} is fulfilled by 
\[
S=\{L\in\R\mid e^L\in\op{Spec}(\partial X,\lambda)\},
\]
by Proposition~\ref{prop:stability}.

If moreover $(X,\lambda)$ is nondegenerate, then we also have
\begin{equation}
\label{eqn:DeltaCH}
\Delta_L(\mathring{CH}(X,\lambda)) = \Delta_L(CH(X,\lambda)),
\end{equation}
by Proposition~\ref{prop:rhoL}.
\end{example}

\begin{definition}
\label{def:zetadt}
If $(X,\lambda)$ is a dynamically tame Liouville domain, define
\begin{equation}
\label{eqn:zetamathring}
\zeta_{\mathring{CH}}(X,\lambda) = \sum_{L\in\R}\Delta_L \left(\mathring{CH}(X,\lambda)\right)t^L \in \Lambda.
\end{equation}
This is a well defined element of the Novikov ring $\Lambda$, because the set of $L$ such that $\Delta_L(\mathring{CH}(X,\lambda))\neq 0$ is closed and discrete and bounded from below. Define
\begin{equation}
\label{eqn:zetadt}
\zeta(X,\lambda) = \Theta\left(\zeta_{\mathring{CH}}(X,\lambda)\right) \in 1+\Lambda^+,
\end{equation}
where $\Theta:\Lambda\to 1+\Lambda^+$ is the function in Definition~\ref{def:Thetadefinition}.
\end{definition}

\begin{example}
If $(X,\lambda)$ is nondegenerate, then the two definitions of $\zeta(X,\lambda)$ in Definition~\ref{def:zetanondegenerate} and \ref{def:zetadt} agree, by equation \eqref{eqn:DeltaCH} and Lemma~\ref{lem:algebraic}.
\end{example}

\begin{proof}[Proof of Theorem~\ref{thm:tame}.]
Let $(X,\lambda)$ and $(X',\lambda')$ be dynamically tame Liouville domains, and suppose there exists an exact symplectomorphism $\phi:(\op{int}(X),d\lambda) \to (\op{int}(X'),d\lambda')$. By Remark~\ref{rem:mathringnotpersistent}, we have $\mathring{CH}^L(X,\lambda) \simeq \mathring{CH}^L(X',\lambda')$ for all $L$. By Definition~\ref{def:ecdt}, if $e^L\notin\op{Spec}(\partial X,\lambda) \cup \op{Spec}(\partial X',\lambda')$, then $\chi^L(X,\lambda) = \chi^L(X',\lambda')$. Then $\Delta_L(\mathring{CH}(X,\lambda)) = \Delta_L(\mathring{CH}(X',\lambda'))$ for all $L\in\R$ as in Definition~\ref{def:tameecj}.  By Definition~\ref{def:zetadt}, we conclude that $\zeta(X,\lambda) = \zeta(X',\lambda')$.
\end{proof}



\section{Computations}
\label{sec:computations}

We now prove Propositions~\ref{prop:toric} and \ref{prop:s1inv}.

\begin{lemma}
\label{lem:toric}
Let $X_\Omega$ be a star-shaped toric domain as in Proposition~\ref{prop:toric}. If $e^L\notin\op{Spec}(\partial X_\Omega)$, then the Euler characteristic in Definition~\ref{def:ecdt} is given by
\begin{equation}
\label{eqn:chilxomega}
\chi^L(X_\Omega) = \floor{\frac{e^L}{a}} + \floor{\frac{e^L}{b}}.
\end{equation}
\end{lemma}

\begin{proof}
We first recall a formula for the Reeb vector field on $\partial X_\Omega$.
Let $\mu:\C^2\to \R^2_{\ge 0}$ be the moment map sending $(z_1,z_2)\mapsto \pi(|z_1|^2,|z_2|^2)$, so that $X_\Omega = \mu^{-1}(\Omega)$. If $(w_1,w_2)\in\op{int}(\partial_+\Omega)$, then $\mu^{-1}(w_1,w_2)$ is a two-torus in $\partial X_\Omega$. Let $(v_1,v_2)$ be an outward normal vector to $\partial_+\Omega$ at $(w_1,w_2)$. Then by the calculations in \cite[\S2.2]{gh}, the Reeb vector field on the torus $\mu^{-1}(w_1,w_2)$ is given by
\[
R = \frac{2\pi}{v_1w_1+v_2w_2}\left(v_1\frac{\partial}{\partial\theta_1} + v_2\frac{\partial}{\partial\theta_2}\right).
\]
Here $\theta_i$ denotes the argument of $z_i$. In particular, $R$ is tangent to the torus $\mu^{-1}(w_1,w_2)$. Similarly, $\mu^{-1}(a,0)$ and $\mu^{-1}(0,b)$ are circles to which the Reeb vector field is tangent and given by $\frac{2\pi}{a}\frac{\partial}{\partial\theta_1}$ and $\frac{2\pi}{b}\frac{\partial}{\partial\theta_2}$ respectively.

It follows that the Reeb orbits on $\partial X_\Omega$ are described as follows:
\begin{description}
\item{(i)}
$\alpha=\mu^{-1}(a,0)$ is a simple Reeb orbit with symplectic action $a$. If $\partial_+\Omega$ has irrational slope at $(a,0)$, then $\alpha$ and all of its iterates are elliptic.
\item{(ii)}
Likewise, $\beta=\mu^{-1}(0,b)$ is a simple Reeb orbit with symplectic action $b$. If $\partial_+\Omega$ has irrational slope at $(0,b)$, then $\beta$ and all of its iterates are elliptic.
\item{(iii)}
If $(w_1,w_2)\in\op{int}(\partial_+\Omega)$, and if $\partial_+\Omega$ has rational slope at $(w_1,w_2)$, let $(v_1,v_2)$ be the outward normal vector to $\partial_+\Omega$ at $(w_1,w_2)$, normalized so that $v_1,v_2$ are relatively prime integers. Then the torus $\mu^{-1}(w_1,w_2)$ is foliated by simple Reeb orbits $\gamma$, each of which has symplectic action
\begin{equation}
\label{eqn:toricaction}
\mathcal{A}(\gamma) = v_1w_1 + v_2w_2.
\end{equation}
If $(w_1,w_2)$ is not an inflection point of the curve $\partial_+\Omega$, then this is a Morse-Bott circle of Reeb orbits as in \cite{bourgeois}. One can then perturb $\partial X_\Omega$ in a neighborhood of $\mu^{-1}(w_1,w_2)$ so that this circle of simple Reeb orbits becomes an elliptic simple Reeb orbit (with the linearized return map $P_\gamma$ a slightly positive rotation or a slightly negative rotation, depending on whether $\Omega$ is locally convex or locally concave at $(w_1,w_2)$) and a positive hyperbolic simple Reeb orbit, possibly together with additional simple Reeb orbits of much larger symplectic action.
\end{description}

Suppose now that $e^L\in\R\setminus\op{Spec}(\partial X_\Omega)$. In particular, note that $e^L$ is not an integer multiple of $a$ or $b$. We now construct an $L$-approximation $\widehat{X}$ as follows (see Definition~\ref{def:Lapprox}). First, we replace $\Omega$ by a new domain $\widehat{\Omega}$ as follows. To start, we scale $\Omega$ slightly to obtain $\widehat{\Omega}$ with $\widehat{\Omega}\subset\op{int}(\Omega)$. Next, if $\partial_+\widehat{\Omega}$ has rational slope at either endpoint, or if it has any inflection points with rational slope, we can correct this by slightly rotating $\widehat{\Omega}$, and then translating it so that its endpoints are on the axes. (The fact that we can eliminate inflection points with rational slope this way follows from Sard's theorem applied to the slope function on $\partial_+\Omega$.) Now the Reeb orbits of $\partial X_{\widehat{\Omega}}$ consist of two elliptic orbits $\widehat{\alpha}$ and $\widehat{\beta}$ as in (i) and (ii) above with symplectic actions $\widehat{a}$ and $\widehat{b}$ respectively, where $\widehat{a}$ is slightly less than $a$ and $\widehat{b}$ is slightly less than $b$, together with the Morse-Bott circles of Reeb orbits described in (iii) above. It follows from equation \eqref{eqn:toricaction} that there are only finitely many such Morse-Bott circles of simple Reeb orbits such that the corresponding simple Reeb orbits have symplectic action less than or equal to $e^L$. We now define $\widehat{X}$ by perturbing $X_{\widehat{\Omega}}$ as in (iii) to replace each of these finitely many circles of simple Reeb orbits by two simple Reeb orbits of symplectic action\footnote{To be precise, we should also choose the perturbation so that for each positive integer $d$, the $d$-fold iterates of these two simple orbits either both have action $< e^L$ or both have action $>e^L$.} less than $e^L$ (one positive hyperbolic and one elliptic), possibly together with additional simple Reeb orbits of symplectic action greater than $e^L$. 

By Proposition~\ref{prop:approx},
\begin{equation}
\label{eqn:chil2}
\chi^L(X_\Omega) = \chi^L(\widehat{X}).
\end{equation}
The Reeb orbits in $\partial\widehat{X}$ with symplectic action less than or equal to $e^L$ arising from the Morse-Bott circles above cancel out in $\chi^L(\widehat{X})$, because elliptic orbits have positive Lefschetz sign while positive hyperbolic orbits have negative Lefschetz sign. Thus $\chi^L(\widehat{X})$ counts, with positive sign, the number of iterates of $\widehat{\alpha}$ or $\widehat{\beta}$ with symplectic action less than or equal to $e^L$. It follows that
\begin{equation}
\label{eqn:chil3}
\chi^L(\widehat{X}) = \floor{\frac{e^L}{\widehat{a}}} + \floor{\frac{e^L}{\widehat{b}}}
\end{equation}
Since $\widehat{a}$ is slightly less than $a$ and $\widehat{b}$ is slightly less than $b$ and $e^L$ is not an integer multiple of $a$ or $b$, equations \eqref{eqn:chil2} and \eqref{eqn:chil3} imply equation \eqref{eqn:chilxomega}.
\end{proof}

\begin{proof}[Proof of Proposition~\ref{prop:toric}.]
It follows from Lemma~\ref{lem:toric} that $X_\Omega$ is dynamically tame, with $S$ consisting of the numbers $\log(da)$ and $\log(db)$ where $d$ is a positive integer.
By equations \eqref{eqn:zetamathring} and \eqref{eqn:chilxomega}, we have
\[
\zeta_{\mathring{CH}}(X_\Omega) = \sum_{d=1}^\infty t^{\log(da)} + \sum_{d=1}^\infty t^{\log(db)}.
\]
By equation \eqref{eqn:zetadt},
\[
\zeta(X_\Omega) = \Theta\left(\sum_{d=1}^\infty t^{\log(da)} + \sum_{d=1}^\infty t^{\log(db)}\right).
\]
By the definition of $\Theta$ in Definition~\ref{def:Thetadefinition} together with Remark~\ref{rem:multiplicative} and equation \eqref{eqn:algebraic1}, it follows that equation \eqref{eqn:zetatoric} holds.
\end{proof}

\begin{proof}[Proof of Proposition~\ref{prop:s1inv}.]
Similarly to the proof of Proposition~\ref{prop:toric}, it is enough to show that if $e^L\notin\op{Spec}(\partial X_f)$, then
\begin{equation}
\label{eqn:chilxf}
\chi^L(X_f) = \sum_{p\in\op{Crit}(f)}(-1)^{\op{ind}(p)}\floor{e^{L-f(p)}}.
\end{equation}
This will imply that $X_f$ is dynamically tame, with $S$ consisting of the numbers $\log(d e^{f(p)})$ where $d$ is a positive integer and $p$ is a critical point of $f$. Then equation \eqref{eqn:algebraic1} will give the desired formula \eqref{eqn:zetaXf} for the zeta function.

Again, we can use Proposition~\ref{prop:approx} to compute $\chi^L(X_f)$ as
\[
\chi^L(X_f) = \chi^L(\widehat{X})
\]
where $\widehat{X}$ is a suitable $L$-approximation of $X_f$. We proceed in three steps.

{\em Step 1.\/}
We begin by computing the Reeb vector field on $\partial X_f$.

Let $\rho:\partial X_f\to \C P^1$ denote the map sending $z\mapsto [z_1:z_2]$. Note that $\partial X_f$ is identified with $S^3$ by the map sending $z\mapsto z/|z|$, and under this identification $\rho$ corresponds to the Hopf fibration. Let $\lambda_0$ denote the contact form on $S^3=\{z\in\C^2 \mid |z|=1\}$, and let $R_0$ denote the associated Reeb vector field. Then $\lambda_0$ is a connection 1-form on the Hopf fibration which integrates to $2\pi$ on each fiber, and $R_0$ is tangent to the fibers, given by the derivative of the $S^1$ action.  Under the identification $\partial X_f\simeq S^3$, the contact form on $\partial X_f$ is given by
\[
\lambda = \frac{1}{2\pi}e^{\rho^*f}\lambda_0.
\]

Let $\omega$ denote the curvature form on $\C P^1$ determined by the connection $\lambda_0$. Let $V_f$ denote the Hamiltonian vector field on $\C P^1$ determined by the function $f$ and the symplectic form $\omega$. Let $\widetilde{V}_f$ denote the vector field on $S^3$ given by the horizontal lift of the $V_f$ using the connection $\lambda_0$. Then the Reeb vector field associated to $\lambda$ is
\begin{equation}
\label{eqn:XfReeb}
R = 2\pi e^{-\rho^*f}\left(R_0 + \widetilde{V}_f\right).
\end{equation}

If $p\in\C P^1$ is a critical point of $f$, then it follows from equation \eqref{eqn:XfReeb} that $\rho^{-1}(p)\subset \partial X_f$ is (the image of) a simple Reeb orbit of symplectic action $e^{f(p)}$. If $p$ has Morse index 1, then this Reeb orbit is positive hyperbolic, while if $p$ has Morse index 0 or 2, then this Reeb orbit is elliptic or degenerate. 

{\em Step 2.\/} To complete the proof of \eqref{eqn:chilxf} and thus of the proposition, it is enough to show that the $L$-approximation $\widehat{X}$ can be chosen so that:
\begin{description}
\item{(i)}
The simple Reeb orbits arising from index $0$ and $2$ critical points of $f$, and their iterates with action less than or equal to $e^L$, are nondegenerate and elliptic.
\item{(ii)}
The contributions to $\chi^L(\widehat{X})$ from any Reeb orbits with action less than or equal to $e^L$ not arising from critical points of $f$ cancel out.
\end{description}
To prepare for this, we now study the Reeb flow on $\partial X_f$ more carefully.

Let $Z\subset\C P^1$ denote the union, over critical points $p$ of $f$, of the connected component containing $p$ of the level set $f^{-1}(f(p))$. Then it follows from equation \eqref{eqn:XfReeb} that the Reeb vector field $R$ is tangent to $\rho^{-1}(Z)$, but there are no Reeb orbits in $\rho^{-1}(Z)$ other than the Reeb orbits arising from the critical points of $f$.

Now let $\Sigma\subset \C P^1$ be a connected component of $\C P^1\setminus Z$. We can identify
\[
\Sigma \simeq (\R/2\pi\Z) \times (y_0,y_1)
\]
such that, letting $x,y$ denote the coordinates on the right hand side, we have $f(x,y)=y$. For $y\in(y_0,y_1)$, let $S_y$ denote the circle $(\R/2\pi\Z)\times\{y\}\subset \Sigma$. Define $h(y)$ to be the $\omega$-area of the connected component of  $\C P^1\setminus S_y$ containing $(\R/2\pi\Z)\times(y_0,y)$. Note that $h(y)\mod 2\pi$ is the holonomy of the connection $\lambda_0$ around the circle $S_y$. Also, $S_y$ is a periodic orbit of the Hamiltonian vector field $V_f$ whose period is the derivative $h'(y)$.

From equation \eqref{eqn:XfReeb} we now read off the following:
\begin{itemize}
\item
The Reeb vector field $R$ is tangent to the torus $\rho^{-1}(S_y)$.
\item
If $\frac{h'(y)+h(y)}{2\pi}$ is irrational, then there are no Reeb orbits in $\rho^{-1}(S_y)$.
\item
If
\[
\frac{h'(y)+h(y)}{2\pi} = \frac{a}{b}
\]
where $a$ and $b$ are relatively prime positive integers, then the torus $\rho^{-1}(S_y)$ is foliated by simple Reeb orbits $\gamma$, each of which has symplectic action
\[
\mathcal{A}(\gamma) = \frac{e^y}{2\pi}bh'(y).
\]
Moreover, if
\begin{equation}
\label{eqn:secondderivative}
h''(y)+h'(y)\neq 0
\end{equation}
then this is a Morse-Bott circle of Reeb orbits.
\end{itemize}

{\em Step 3.\/} We now explain how to choose the $L$-approximation $\widehat{X}$ satisfying (i) and (ii) above.

To start, without indicating this in the notation, we scale $f$ down slightly to obtain a subset of the interior of the original $X_f$. 

To obtain (i), it is sufficient if in the situation of Step 2, the following holds: If $S_y$ converges to an index $0$ critical point as $y\searrow y_0$, then $\lim_{y\searrow y_0}\frac{h'(y)+h(y)}{2\pi}$ is irrational; and if $S_y$ converges to an index $2$ critical point as $y\nearrow y_1$, then $\lim_{y\nearrow y_1}\frac{h'(y)+h(y)}{2\pi}$ is irrational. We can effect this by adding small quadratic terms to $f$ near its index $0$ and $2$ critical points.

To obtain (ii), it is sufficient if in the situation of Step 2, whenever $\frac{h'(y)+h(y)}{2\pi}$ is rational (with some upper bound on the denominator), condition \eqref{eqn:secondderivative} holds. One can then perturb $\partial X_f$ so that the corresponding Morse-Bott circle of Reeb orbits splits into two simple Reeb orbits, as in the toric case in Proposition~\ref{prop:toric}, so that the contributions of these simple orbits and their iterates to $\chi^L$ cancel, and any additional orbits created have action greater than $e^L$. To obtain condition \eqref{eqn:secondderivative} where required, in the situation of Step 2, one can modify $f$ so that away from the endpoints of the interval $(y_0,y_1)$, a small constant is added to $h$. One can then invoke Sard's theorem as in the proof of Proposition~\ref{prop:toric}.
\end{proof}



\end{document}